\documentclass[12pt,oneside,reqno]{amsart} 
\usepackage{etex}
\usepackage[T2A]{fontenc}
\usepackage[cp1251]{inputenc}
\usepackage{pifont}
\usepackage[dvips]{graphicx}
\usepackage{amsmath,latexsym,amsthm,cmap,indentfirst,setspace,ccaption,amssymb,color,amscd,mathrsfs,hyperref}
\usepackage{array}
\usepackage[normalem]{ulem}
\usepackage{fourier}

\usepackage{amssymb,lipsum}
\usepackage{float,xypic}
\usepackage{cmap,longtable}

\usepackage{geometry}
\usepackage{caption, subcaption}

\usepackage{tikz}

\usepackage{verbatim}
\hypersetup{
    bookmarks=true,         
    unicode=true,          
    pdftoolbar=true,        
    pdfmenubar=true,        
    pdffitwindow=false,     
    pdfnewwindow=true,      
    colorlinks=true,       
    linkcolor=blue,          
    citecolor=blue,        
    filecolor=magenta,      
    urlcolor=cyan           
}
\usepackage[backend=bibtex, style=alphabetic, citestyle=alphabetic]{biblatex}
\bibliography{BG}
\AtEveryBibitem{\clearfield{issn}}
\AtEveryBibitem{\clearlist{issn}}

\AtEveryBibitem{\clearfield{language}}
\AtEveryBibitem{\clearlist{language}}

\AtEveryBibitem{\clearfield{doi}}
\AtEveryBibitem{\clearlist{doi}}

\AtEveryBibitem{\clearfield{fjournal}}
\AtEveryBibitem{\clearlist{fjournal}}

\AtEveryBibitem{%
  \ifentrytype{online}
    {}
    {\clearfield{urlyear}\clearfield{urlmonth}\clearfield{urlday}}}

\newcommand{\p}{\mathbb{P}}
\DeclareMathOperator{\tr}{tr}
\DeclareMathOperator{\Bir}{Bir}

\DeclareMathOperator{\Aut}{Aut}
\DeclareMathOperator{\Sym}{Sym}

\DeclareMathOperator{\GL}{GL}

\DeclareMathOperator{\Bound}{B}

\DeclareMathOperator{\Ort}{O}
\DeclareMathOperator{\Mer}{\mathcal M}

\DeclareMathOperator{\Bim}{Bim}
\DeclareMathOperator{\Jord}{J}

\DeclareMathOperator{\Exc}{Exc}

\DeclareMathOperator{\SG}{\mathfrak{S}}
\DeclareMathOperator{\Z}{\mathcal{Z}}
\DeclareMathOperator{\E}{\mathcal{E}}

\renewcommand{\U}{\mathcal{U}}
\DeclareMathOperator{\V}{\mathcal{V}}
\renewcommand{\H}{\mathcal{H}}
\DeclareMathOperator{\N}{\mathcal{N}}
\DeclareMathOperator{\D}{\mathcal{D}}
\DeclareMathOperator{\OP}{\mathcal{O}}
\renewcommand{\S}{\mathcal{S}}
\renewcommand{\L}{\mathcal{L}}
\DeclareMathOperator{\Pcal}{\mathcal{P}}

\makeatletter
\newcommand{\xleftrightarrow}[2][]{\ext@arrow 3359\leftrightarrowfill@{#1}{#2}}
\newcommand{\xdashrightarrow}[2][]{\ext@arrow 0359\rightarrowfill@@{#1}{#2}}
\newcommand{\xdashleftarrow}[2][]{\ext@arrow 3095\leftarrowfill@@{#1}{#2}}
\newcommand{\xdashleftrightarrow}[2][]{\ext@arrow 3359\leftrightarrowfill@@{#1}{#2}}
\def\rightarrowfill@@{\arrowfill@@\relax\relbar\rightarrow}
\def\leftarrowfill@@{\arrowfill@@\leftarrow\relbar\relax}
\def\leftrightarrowfill@@{\arrowfill@@\leftarrow\relbar\rightarrow}
\def\arrowfill@@#1#2#3#4{%
  $\m@th\thickmuskip0mu\medmuskip\thickmuskip\thinmuskip\thickmuskip
   \relax#4#1
   \xleaders\hbox{$#4#2$}\hfill
   #3$%
}
\makeatother

\theoremstyle{plain}
\newtheorem*{theoremA}{Theorem A}
\newtheorem*{theoremB}{Theorem B}
\newtheorem*{theoremC}{Theorem C}

\renewcommand{\a}{\mathrm{a}}
\theoremstyle{plain}
\newtheorem{theorem}{Theorem}[section]
\newtheorem{lemma}[theorem]{Lemma} 
\newtheorem{corollary}[theorem]{Corollary} 
\newtheorem{proposition}[theorem]{Proposition}

\theoremstyle{definition}
\newtheorem{definition}[theorem]{Definition} 
\newtheorem{remark}[theorem]{Remark}
\newtheorem*{question}{Question}
\newtheorem{example}[theorem]{Example}

\makeatletter
\g@addto@macro{\endabstract}{\@setabstract}
\newcommand{\authorfootnotes}{\renewcommand\thefootnote{\@fnsymbol\c@footnote}}%
\makeatother	

\newcommand\blfootnote[1]{%
	\begingroup
	\renewcommand\thefootnote{}\footnote{#1}%
	\addtocounter{footnote}{-1}%
	\endgroup
}

\author{Fedor Bogomolov}
\address{Courant Institute, New York University, 251 Mercer Str.
	New York, USA 10012. Also: Laboratory of Algebraic Geometry,
	National Research University HSE,
	Department of Mathematics, 6 Usacheva Str. Moscow, Russia}
\email{bogomolo@cims.nyu.edu}

\author{Nikon Kurnosov}
\address{Department of Mathematics,	University of Georgia,
	Athens, GA, USA, 30602.	Also: Laboratory of Algebraic Geometry,
	National Research University HSE,
	Department of Mathematics, 6 Usacheva Str. Moscow, Russia}
\email{nikon.kurnosov@gmail.com}

\author{Alexandra Kuznetsova}
\address{Ecole Polytechnique, CMLS,
	France, Route de Saclay, 91128 Palaiseau.
	Also: Laboratory of Algebraic Geometry,
	National Research University HSE,
	Department of Mathematics, 6 Usacheva Str. Moscow, Russia}
\email{sasha.kuznetsova.57@gmail.com}

\author{Egor Yasinsky}
\address{Universit{\"a}t Basel,
	Departement Mathematik und Informatik,
	Spiegelgasse 1,
	CH-4051, Basel,
	Switzerland}
\email{yasinskyegor@gmail.com}

\begin{document}
	
\newcommand{\ZZ}{\mathbb Z}
\newcommand{\CC}{\mathbb C}
\newcommand{\PP}{\mathbb P}
\newcommand{\Torus}{\mathbb T}
\newcommand{\Sph}{\mathbb S}
\newcommand{\Line}{\mathcal L}
\newcommand{\EEE}{\mathcal E}
\newcommand{\OOO}{\mathcal O}
\newcommand{\GGG}{\mathcal G}

\newcommand{\Base}{W_F}
\newcommand{\Poincare}{\mathcal P}
\newcommand{\id}{\mathrm{id}}
	
\begin{center}
	\LARGE 
	Geometry and automorphisms of non-K\"{a}hler holomorphic symplectic manifolds \par \bigskip\bigskip
	
	\normalsize
	\authorfootnotes
	\large Fedor Bogomolov, Nikon Kurnosov, Alexandra Kuznetsova and Egor Yasinsky \par \bigskip

\end{center}

\begin{abstract}
We consider the only one known class of non-K\"ahler irreducible holomorphic symplectic manifolds, described in the works of D.~Guan and the first author. 
Any such manifold $Q$ of dimension \mbox{$2n-2$} is obtained as a finite degree $n^2$ cover of some non-K\"ahler manifold $\Base$ which we call the {\it base} of~$Q$. 
We show that the algebraic reduction of~$Q$ and its base is the projective space of dimension~\mbox{$n-1$}. Besides, we give a partial classification of submanifolds in $Q$, 
describe the degeneracy locus of its algebraic reduction and prove that the automorphism group of $Q$ satisfies the Jordan property.
\end{abstract}

\section{Introduction}
\label{sec: introduction}
\subsection{Holomorphic symplectic manifolds}
Let\blfootnote{2020 {\it Mathematics Subject Classification}. 14J42, 53D05, 53C26, 32C15, 32C25, 32H04, 32Q99, 32M18, 14K99.} $M$ be a complex manifold of dimension $2n$. A {\it holomorphic symplectic structure} is a closed holomorphic 2-form $\omega_M$ on $M$ of maximal rank. 
One can ask when a compact holomorphic symplectic manifold is K\"{a}hler, i.e. admits a positive closed $(1,1)$-form. By the Enriques-Kodaira classification, 
in dimension two there are 3 different types of holomorphic symplectic surfaces: K3 surfaces, complex tori and Kodaira-Thurston surfaces (see below). 
Every simply connected holomorphic symplectic manifold of dimension 2 is a K3 surface, and those are known to be K\"{a}hler by the result of Siu \cite{Siu}. 
It then was conjectured by A.~Todorov that this should generalize to higher dimensions. In fact, in his MPIM preprint \cite{Todorov} Todorov claimed that every 
holomorphic symplectic manifold admits a K\"{a}hler metric. However, ten years later several counterexamples to this statement were found by D. Guan \cite{Guan1,Guan3,Guan2}. 
A more geometrically transparent construction of these manifolds was then given by the first author in \cite{Bogomolov}. 
In the present paper we call\footnote{In \cite{Kur-Ver} this class of manifolds was called Bogomolov-Guan manifolds.} them {\it BG-manifolds}. 
Their construction is sketched in Section \ref{subsec: BG construction}.

Holomorphic symplectic manifolds are closely related to {\it hyperk\"{a}hler} manifolds. Recall that a hyperk\"{a}hler manifold is a Riemannian manifold $(M,g)$ 
equipped with three K\"{a}hler complex structures $I,J,K: TM\to TM$, satisfying the quaternionic relation
\[
I^2=J^2=K^2=IJK=-\id.
\]
Any hyperk\"{a}hler manifold is holomorphically symplectic. Indeed, a simple computation shows that the form 
\[
\omega=\omega_J+\sqrt{-1}\omega_K
\]
is of type $(2,0)$ (here $\omega_*=g(*X,Y)$). Moreover, it is closed and hence holomorphic. 
Conversely, a compact holomorphically symplectic manifold is hyperk\"{a}hler, provided that it is K\"{a}hler. This follows from the Calabi-Yau theorem \cite{Yau}. 

One of the most important properties of hyperk\"ahler manifolds is the existence of the Beauville-Bogomolov-Fujiki-form (BBF-form, for short), 
that makes the second cohomology group into lattice. This fact is the consequence of the Bogomolov-Tian-Todorov theorem, 
which says that the deformation theory of K\"ahler manifolds with trivial canonical class is unobstructed. 
In the recent work \cite{Kur-Ver} it was proven that the deformation theory of BG-manifolds is similar to that of hyperk\"ahler manifolds. 
More precisely, there is a version of local Torelli theorem which allows to show that holomorphic symplectic deformations of BG-manifolds are unobstructed, 
and the corresponding period map is locally an isomorphism. Moreover, a BG-manifold $M$ of dimension $\dim_\CC M=2n$ possesses a non-K\"ahler version of the BBF-form, 
i.e. there exists a symmetric form $q$ on $H^2(M)$ such that for any $w\in H^2(M)$ one has
\[
\int_M w^{2n}=q(w,w)^n.
\]
Conjecturally, this form is also non-degenerate.

\subsection{Algebraic reduction and submanifolds of BG-manifolds} The goal of this article is to shed some light on the geometry of non-K\"{a}hler manifolds, described in the works of Guan and the first author, and also some related manifolds. As described in Section \ref{subsec: BG construction}, any BG-manifold $Q$ of dimension $2n-2$ is obtained as a finite degree $n^2$ cover of some complex non-K\"{a}hler manifold $\Base$ which we call the {\it base} of $Q$. The first main result of this article describes the algebraic reduction of $Q$ and $\Base$.

\begin{theoremA}\label{thm: main A}
	Let $n\geqslant 2$ be an integer, $Q$ be a BG-manifold of dimension $2n-2$, and $\Base$ be its base. Then $Q$ and $\Base$ have a structure of a fiber space
	\[
	\Phi: Q\to\PP^{n-1},\ \ \ \Pi: \Base\to\PP^{n-1}
	\]
	with typical fiber being an abelian variety; these maps are the algebraic reductions of $Q$ and $\Base$, respectively. In particular, the algebraic dimension of $Q$ equals $n-1$.
\end{theoremA}

The construction of BG-manifolds actually reminds the construction of (hyperk\"ahler) generalized Kummer varieties, but starts with the Kodaira-Thurston surface, rather than a complex torus. The idea to start with Kodaira surface to produce new examples of irreducible holomorphic manifolds is based on the fact that Kodaira surface is one of five possible complex surfaces with Kodaira dimension zero together with K3 surfaces, abelian surfaces, Enriques surfaces and hyperelliptic surfaces \cite[VI]{BPV}, and it is the only non-K\"ahler one among those. If $S$ is a Kodaira-Thurston surface, then its algebraic reduction
\[
\pi: S\to E
\]
is a principal elliptic fibration over an elliptic curve $E$. It is well known that all algebraic submanifolds of $S$ are contained in the fibers of $\pi$. The second main result of this paper gives some information about the properties of submanifolds in BG-manifolds (again using the algebraic reduction $\Phi$ of $Q$). A somewhat more precise description of their geometry can be extracted from Proposition \ref{prop: submanifolds of Q} and Theorem \ref{thm_classification_for_Sn/F}.

\begin{theoremB}\label{thm: submanifolds of BG}
Let $Q$ be a BG-manifold and $X\subset Q$ be its submanifold. Set $Z = \Phi(X)$.
\begin{enumerate}
			\item If $Z$ is a point, then $X$ is a subvariety of the fiber $\Phi^{-1}(Z)$ which is algebraic for a general point $Z$ and Moishezon for others.
			\item If $Z$ is a smooth curve, then $X$ can be Moishezon or non-Moishezon, depending on $Z$ (see Proposition~\ref{prop: submanifolds of Q} for a more precise statement).
			\item If $\dim(Z) \geqslant 2$, then $X$ is not Moishezon.
		\end{enumerate} 
\end{theoremB}

\subsection{Automorphisms}
Given a complex manifold $X$, one is often interested in the properties of its group of biholomorphic or bimeromorphic transformations, denoted in this article by $\Aut(X)$ and $\Bim(X)$ respectively (or $\Bir(X)$ for $X$ algebraic). In general, these groups may have a  very complicated structure and hence one can try to study their finite subgroups instead. In fact, it often happens that one can observe some kind of boundedness on this level. Following \cite{Pop2}, we say that a group $\Gamma$ is {\it Jordan} 
if there exists a constant $\Jord=\Jord(\Gamma)\in\ZZ_{>0}$ such that every finite subgroup $G\subset\Gamma$ has a normal abelian subgroup $A\subset G$ with $[G:A]\leqslant \Jord$. The minimal such $\Jord$ is called the {\it Jordan constant} of $\Gamma$ and is denoted $\Jord(\Gamma)$. 
A classical theorem due to Camille Jordan says that $\GL_n(\CC)$ enjoys this property, hence the name. 

In recent years, the property of being Jordan has been investigated in several different contexts, and many groups of geometric nature were observed to satisfy it, for example: diffeomorphism groups of some smooth manifolds \cite{Mun-Diff, Mun-4folds,Mun-Ham,Mun-Diff,Zim}; groups of birational selfmaps of rationally connected \cite{PS-Cremona,Birkar} and of non-uniruled \cite{PS-Jordan} algebraic varieties; in fact, all {\it biregular} automorphism groups of complex {\it projective} varieties \cite{Meng} and even just K\"{a}hler ones \cite{Kim}. A somewhat sporadic example is the one of three-dimensional Moishezon compact complex spaces \cite{PS-Moish}. 

\begin{remark}\label{rem: Zarhin-Popov}
	There is a complete classification of complex projective surfaces and threefolds whose groups of birational transformations are not Jordan. It was proved by V.~Popov that $\Bir(X)$ is Jordan for all algebraic surfaces which are not birational to $\PP^1\times E$, where $E$ is an elliptic curve \cite{Pop2}. Then Yu. Zarhin showed that $\Bir(A\times Z)$ is not a Jordan for any abelian variety $A$ and rational variety $Z$ \cite{Zarh}. For threefold case see \cite{PS-3folds}.
\end{remark}

The aforementioned result of Popov and Zarhin can be actually generalized to all complex compact surfaces, including non-K\"ahler ones (e.g. Kodaira-Thurston surface), see \cite{PS-Compact}. However, essentially nothing is known about the Jordan property of higher dimensional non-K\"ahler manifolds. In this paper we take an attempt to get some improvements in that direction and establish the following (for definitions used in the statement see Section \ref{sec: auto}):

\begin{theoremC}\label{main thm: Jordan}
	Let $Q$ be a BG-manifold of dimension $2n-2$. Then
	\begin{enumerate}
		\item The group $\Aut(Q)$ of its biholomorphic automorphisms is (strongly) Jordan.
		\item Denoting by $\Bim(Q)_{\Phi}\subset\Bim(Q)$ the subgroup of those bimeromorphic transformations which are fiberwise with respect to the algebraic reduction $\Phi$ of $Q$, we have a short exact sequence
		\[
		1\to\Bim(Q)_{\Phi}\to\Bim(Q)\to\Delta\to 1,
		\] 
		where $\Delta$ is a subgroup of the Cremona group $\Bir(\PP^{n-1})$, and the group $\Bim(Q)_{\Phi}$ is (strongly) Jordan. The group $\Bim(Q)$ is quasi-bounded.
	\end{enumerate}
\end{theoremC}

However the following question remains open (see also Remark \ref{rem: boundednes of Aut}):

\begin{question}
	Does $\Bim(Q)$ satisfy the Jordan property? 
\end{question}

This paper is organized as follow. In Section \ref{sec: preliminaries} we list some well-known definitions and properties related to holomorphic symplectic and complex analytic manifolds. In Section \ref{sec: BG} we recall the construction of BG-manifolds (mostly following the approach of the first author). Every such manifold $Q$ is a finite cover of some manifold $\Base$ called the {\it base} of $Q$. In Section \ref{sec: W/F} we prove Theorem A, computing the algebraic reduction of $\Base$, and then of $Q$ itself. Section \ref{sec: class N} is auxiliary and introduces some useful filtration on the set of all complex spaces, that we shall use in the sequel. Section \ref{sec: submanifolds of BG} is devoted to description of submanifolds of $Q$ and the proof of Theorem B. Finally, in Section \ref{sec: degenerate fibers} we describe the discriminant locus of the algebraic reduction of $Q$ and its degenerate fibers. This allows us to establish the Jordan property of $\Aut(Q)$ in Section \ref{sec: auto}, proving Theorem C.

\hfill

{\bf Acknowledgement.} F. B. is partially supported by EPSRC programme grant EP/M024830. 
F. B. and N. K.  are partially supported by the HSE University Basic Research Program, Russian Academic Excellence Project '5-100' .
N. K. and A. K. are supported by the "BASIS" foundation,  by the Simons Travel grant. E. Y. acknowledges support by the Swiss 
National Science Foundation Grant ``Birational transformations of threefolds'' 200020\_178807. We are
grateful to Dmitry Kaledin and Vasya Rogov for important observations and helpful remarks; A. K. is also thankful to the
organizers and attendees of the Workshop on Complex and Algebraic Geometry held in Moscow (March 18, 2020) for useful comments and remarks.

\section{Preliminaries}\label{sec: preliminaries}

In this section we briefly recall some main definitions and basic facts about complex manifolds and their bimeromorphic geometry. 

\subsection{Holomorphic symplectic manifolds}

\begin{definition}
	A {\it holomorphic symplectic manifold} is a complex manifold $M$ equipped with a non-degenerate holomorphic $(2,0)$-form $\omega_M$.
\end{definition}

In dimension 2, we have 3 types of holomorphic symplectic surfaces: complex tori, K3 surfaces and Kodaira-Thurston surfaces. 
The latter is of central importance for the present paper. We recall its main properties in Section \ref{subsec: Kodaira}.

\begin{definition}
	A {\it hyperk\"{a}hler manifold} is a Riemannian manifold $(M,g)$ equipped with three K\"{a}hler complex structures $I,J,K: TM\to TM$, satisfying the quaternionic relation
	\[
	I^2=J^2=K^2=IJK=-\id.
	\]
\end{definition}

Since we do not assume our manifolds to be K\"{a}hler, in this paper we have to distinguish between the terms ``hyperk\"{a}hler'' and ``holomorphic symplectic'', which are often used as synonyms in the literature. Any hyperk\"{a}hler manifold is automatically holomorphic symplectic, because
\[
\omega=\omega_J+\sqrt{-1}\omega_K
\]
is a holomorphic symplectic form on $(M,I)$. Now recall the fundamental
\begin{theorem}[Calabi-Yau Theorem {\cite{Yau}}]
	Let $M$ be a compact holomorphic symplectic K\"{a}hler manifold. Then $M$ admits a hyperk\"{a}hler metric, which is uniquely determined by the cohomology class of its K\"{a}hler form. 
\end{theorem} 

Therefore, every {\it K\"{a}hler} holomorphic symplectic manifold is actually hyperk\"{a}hler.

\begin{remark}
	A hyperk\"ahler manifold $M$ is called {\it of maximal holonomy}, or {\it simple}, or {\it irreducible holomorphic symplectic (IHS)}, if~\mbox{$\pi_1(M)=0$}, $H^{2,0}(M)=\CC \left\langle \omega \right\rangle $. Simple hyperk\"ahler manifolds are building blocks of all compact hyperk\"ahler manifolds: according to Bogomolov's decomposition theorem, any hyperk\"ahler manifold admits a finite covering
	which is a product of a torus and several maximal holonomy hyperk\"ahler manifolds. The maximal holonomy hyperk\"ahler components of this decomposition are defined uniquely.
\end{remark}

At the moment, there are several known examples of simple hyperk\"{a}hler manifolds: two sporadic O'Grady examples in dimensions six and ten, Hilbert schemes $Z^{[n]}$ of 0-dimensional closed subschemes of length $n$ of a K3 surface $Z$, and generalized {\it Kummer varieties}, i.e. the kernels of the composition 
\[
T^{[n]}\to\Sym^n{T}\overset{\Sigma}{\to} T,
\]
where $T$ is a complex torus and $\Sigma$ is the sum morphism. As we will see in Section \ref{subsec: BG construction}, the construction of BG-manifolds actually reminds that of generalized Kummer varieties 
(but starts with the Kodaira-Thurston surface, rather than a complex torus).

\subsection{Complex spaces and meromorphic maps}

Our main references for this paragraph are \cite{GPR} and \cite{Ueno}. In this paper, all complex spaces are assumed to be reduced and irreducible and compact, if not stated otherwise. A {\it complex manifold} is a smooth complex space. The term ``variety'' is reserved for algebraic ones (which differs from Ueno's book). 

\begin{definition}\label{def: fiber spaces} Let $X$ and $Y$ be complex spaces.
	\begin{itemize}
		\item A {\it fiber space} is a proper surjective morphism $f:X\to Y$ with irreducible general fiber.
		\item A morphism $f:X\to Y$ is called a {\it (proper) modification}, if $f$ is proper, surjective, and there exist closed analytic subsets $Z_X\subsetneq X$ and $Z_Y\subsetneq Y$ such that $f$ induces an isomorphism $X\setminus Z_X\cong Y\setminus Z_Y$. 
		\item A {\it meromorphic map} $f:X\dashrightarrow Y$ is a map $X\to 2^Y$ such that its graph $\Gamma_f$ is an irreducible closed analytic subset of $X\times Y$, and the projection $p_X: \Gamma_f\to X$ is a proper modification. A meromorphic map is called {\it bimeromorphic} if the projection $p_Y: \Gamma_f\to Y$ is a modification as well. 
		\item A {\it meromorphic fiber space} is a meromorphic map $f: X\dashrightarrow Y$ such that $p_Y$ is a fiber space.
	\end{itemize}
	
\end{definition}

The group of all bimeromorphic maps $X\dashrightarrow X$ will be denoted $\Bim(X)$. 

\begin{remark}
	By GAGA principle, for a smooth complex projective variety $X$ one has $\Bim(X)=\Bir(X)$, the group of birational automorphisms of $X$.
\end{remark}

Given a complex space $X$, its field of meromorphic functions will be denoted $\Mer(X)$. This is a finitely generated extension over $\CC$, satisfying
\begin{equation}\label{eq: alg dimension bound}
0\leqslant\tr\deg_\CC\Mer(X)\leqslant\dim(X),
\end{equation}
see \cite[Theorem 3.1]{Ueno}. The integer number $\a(X)=\tr\deg_{\CC}\Mer(X)$ is called the {\it algebraic dimension} of $X$. Further, $X$ is said to be a {\it Moishezon space} if $\a(X)=\dim(X)$ \cite{M}.

\begin{remark}
	Any bimeromorphic map $f: X\dashrightarrow Y$ induces an isomorphism of the fields of meromorphic functions
	\[
	f^*: \Mer(Y)\overset{\raisebox{0.25ex}{$\sim\hspace{0.2ex}$}}{\smash{\longrightarrow}}\Mer(X),
	\]
	so $\a(X)$ is a bimeromorphic invariant. The converse does not hold in general, but holds for Moishezon spaces \cite[VII, Corollary 6.8]{GPR}.
\end{remark}

\begin{remark}\label{rem: Artin}
	By Artin's theorem, every Moishezon space carries an algebraic space structure.
\end{remark}

Let us also mention the behavior of the algebraic dimension under some maps.
	\begin{lemma}[{\cite[Theorem 3.8]{Ueno}}]\label{lem: alg dimension properties}
		Let $f:X\dashrightarrow Y$ be a surjective meromorphic map of irreducible compact complex spaces. Set
		\[
		\a(f)=\inf_{y\in Y}\a(f^{-1}(y)),\ \ \dim f=\dim X-\dim Y.
		\]
		Then one has
		\[
		\a(Y)\leqslant \a(X)\leqslant \a(Y)+\a(f)\leqslant \a(Y)+\dim f.
		\]
	\end{lemma}

\begin{definition}\label{def: algebraic reduction}
	Given a compact complex space $X$, its {\it algebraic reduction} is a meromorphic fiber space $f:X\dashrightarrow X_0$ to a projective variety $X_0$, such that $f$ induces an isomorphism $\Mer(X_0)\cong\Mer(X)$. 
\end{definition} 

An algebraic reduction is unique up to bimeromorphic equivalence and clearly $\a(X)=\dim(X_0)$. Further, as explained in Definition \ref{def: fiber spaces}, we may assume the algebraic reduction to be a (holomorphic) fiber space. 

\begin{example}
	If $X$ is a surface, then its algebraic reduction is a holomorphic map, whose typical fiber is an elliptic curve \cite[VI.5.1]{BPV}.
\end{example}

We shall need the following result in the future.

\begin{lemma}\label{lemma_alg_red_of_product}
 If $\xi_X \colon X\to X_0$ and $\xi_Y\colon Y\to Y_0$ are the algebraic reductions of compact complex manifolds $X$ and $Y$, then
 \begin{equation*}
  (\xi_X\times \xi_Y)\colon X\times Y \to X_0 \times Y_0
 \end{equation*}
 is the algebraic reduction of $X\times Y$. In particular, $\a(X\times Y) = \a(X)+\a(Y)$. 
\end{lemma}
\begin{proof}
 {It is clear that $\xi_X\times\xi_Y$ is a fiber space.} Consider the map 
 \begin{equation*}
  (\xi_X\times \id_Y)\colon X\times Y \to X_0\times Y.
 \end{equation*}
 Fix a generic point $y\in Y$ and restrict this map to a submanifold $X\times\{y\}$:
 \begin{equation}\label{eq: lemma alg red}
  (\xi_X\times \id_Y)|_{X\times\{y\}}\colon X\times\{y\} \to X_0\times\{y\}.
 \end{equation}
 Now take a meromorphic function $f$ on $X\times Y$. The map (\ref{eq: lemma alg red})  is the algebraic reduction, so the restriction of the function $f|_{X\times\{y\}}$ is constant on fibers of $(\xi_X\times \id_Y)|_{X\times\{y\}}$.
 Then, the function $f$ is constant on all fibers of the map~\mbox{$\xi_X\times \id_Y$}. {Therefore, one can define a map
\[
\varphi: \Mer(X\times Y)\to \Mer(X_0\times Y),\ \ \ \varphi(f)(x_0,y)=f(x,y),\ \ \ \text{where}\ x\in\xi_X^{-1}(x_0),
\] 
which is easily checked to be an isomorphism.
} 
 By repeating this argument 
 for the map~\mbox{$\id_{X_0}\times \xi_Y$} we get the following chain of isomorphisms:
 \begin{equation*}
 \Mer(X\times Y) \cong \Mer(X_0\times Y) \cong \Mer(X_0\times Y_0). 
 \end{equation*}
 Since $X_0\times Y_0$ is a projective variety, it is the algebraic reduction of $X\times Y$.
\end{proof}

\section{Non-K\"ahler irreducible holomorphic symplectic manifolds}\label{sec: BG}

\subsection{Kodaira surfaces}\label{subsec: Kodaira} The known examples of non-K\"ahler irreducible holomorphic symplectic manifolds are constructed from Kodaira-Thurston surfaces, 
so we first recall main properties of the latter ones. Most of the statements of this paragraph can be found in Kodaira's original paper \cite{Kodaira}.

A {\it Kodaira-Thurston surface} (often called just {\it Kodaira surface}) is a compact complex surface of Kodaira dimension 0 with odd first Betti number 
(so, it is never K\"{a}hler). 
There are two kinds of Kodaira surfaces: primary and secondary ones. Every secondary surface is just a quotient of a primary Kodaira 
surface $S$ by a finite cyclic group acting freely on $S$. Any primary Kodaira surface $S$ (we fix this notation until the end of the 
paper) can be constructed as follows. Let $E$ be an elliptic curve, and take a line bundle $\Line$ on $E$ with the first Chern 
class $c_1(\Line)\ne 0$. Denote by $S'$ the total space of $\Line$ with zero section removed (so, $S'$ is a $\CC^*$-bundle on $E$). 
Now fix $\lambda\in\CC$ with $|\lambda|>1$, and let $g_\lambda: S'\to S'$ be the corresponding homothety. Then $S=S'/\langle g_\lambda\rangle$. 

\begin{remark}
	Topologically, $S$ has a structure of a principal $\Sph^1$-fibration over $\Torus^3=\Sph^1\times\Sph^1\times\Sph^1$.
\end{remark}

One can show that $S$ has the following invariants \cite[Table 10]{BPV}:
\[
K_S\sim 0,\ \ a(S)=1,\ \ b_1(S)=3,\ \ b_2(S)=4,\ \ \chi(S)=0,\ \ h^{0,1}(S)=2,\ \ h^{0,2}(S)=1.
\] 
Here, and throughout the paper, $a(S)$ denotes the algebraic dimension of $S$.
We will use the following property of manifolds with a map to Kodaira surface.
\begin{lemma}\label{lemma_map_to_Kodaira}
 Assume that $f: M\dashrightarrow S$ is a bimeromorphic map from a smooth surface $M$ to a Kodaira surface $S$. Then $M$ is non-K\"ahler.
\end{lemma}
\begin{proof}
 If $f$ is not isomorphism in a point $x\in S$, then by \cite[III.4.2]{BPV} it is a $\sigma$-process with center in $x$. By \cite[I.9.1]{BPV} the $\sigma$-process does not change 
 the Betti number~$b_1$. Thus,we get that the number $b_1(M) = b_1(S) = 3$ is odd which is impossible for K\"ahler manifolds.
\end{proof}

\subsection{Bogomolov-Guan example}\label{subsec: BG construction}
Let us recall the first example of simply-connected non-K\"ahler compact complex manifolds, which we call {\it BG-manifolds} as in \cite{Kur-Ver}. 
This example was introduced by D.~Guan in \cite{Guan1} and then studied by the first author in \cite{Bogomolov}. Here we mainly follow Bogomolov's work.

Recall that an algebraic reduction of a compact complex surface $X$ of algebraic dimension 1 is the morphism $X\to C$ to a curve $C$ obtained as regularization of a meromorphic map $X\dashrightarrow\PP^1$ (defined by a non-constant meromorphic function), followed by the Stein factorization. Let 
\begin{equation}\label{eq_Kodaira_projection}
\pi\colon S \xrightarrow{F} E.
\end{equation}
be the algebraic reduction of a Kodaira surface $S$. Then $E$ is an elliptic curve and $\pi$ is a principal elliptic fibration \cite{BPV}.
All fibers of $\pi$ are isomorphic to each other, denote it $F$. 
Moreover, the field $\Mer(S)$ of all meromorphic functions of $S$ is isomorphic to $\Mer(E)$. This implies in particular that all automorphisms of~$S$ are induced by automorphisms of $E$. 

In what follows we denote by $\Sym^n E$ and $\Sym^n S$ the symmetric products $E^n/\SG_n$ and $S^n/\SG_n$.
Denote by $S^{[n]}$ and $E^{[n]}$ the Hilbert schemes of length $n$ zero-dimensional subschemes of $S$ and~$E$ respectively. 
Then we have the induced projection:
\begin{equation*}
\pi^{[n]} = \left(\Sym^n(\pi)\circ\delta\right)\colon S^{[n]} \xrightarrow{F^{[n]}} \Sym^n E.
\end{equation*}
Here $\delta$ is a resolution of singularities $S^{[n]}\to \Sym^n S$. The generic fiber of this projection is 
isomorphic to a $F^n$.  Define
\begin{equation*}
\Sigma:\ \Sym^n E\to E,\ \ \ \ x=\{x_1,\ldots,x_n\}\mapsto x_1+\ldots + x_n.
\end{equation*}
It is the Abel-Jacobi map of $\Sym^n E$. The symmetric product of a smooth curve is smooth; thus, we have an isomorphism $E^{[n]}\cong \Sym^n E$. 
This gives us the following structure of the fiber space:
\begin{equation*}\label{eq_HilbKodaira_projection}
\pi_n =\left(\Sigma \circ \pi^{[n]}\right) \colon S^{[n]} \to E.
\end{equation*}
Fix $0\in E$ and denote by $W$ the fiber of this projection over zero point:
\begin{equation*}
W = \pi_n^{-1}(0) \xrightarrow{\pi^{[n]}} \Sigma^{-1}(0).
\end{equation*}
Abel's theorem implies that the fiber of the Abel-Jacobi map $\Sigma$ through a point $D\in\Sym^n E$ is the projective space $|D|=\PP H^0(\OOO_E(D))$, 
i.e. $\Sigma^{-1}(0)\cong\PP^{n-1}$, see \cite[I.3]{Arbarello}.

The action of $F$ on fibers of $S$ induces a fiberwise diagonal action on $S^{[n]}$, and, therefore, on the fiber space $W$. Thus, we have a projection 
\begin{equation}\label{eq_Pi}
\Pi\colon W/F \to \p^{n-1}.
\end{equation}
The manifold
\[
W_F=W/F
\]
will be called the {\it base} of a BG-manifold $Q$, which is given by the following result of the first author.
\begin{theorem}[{\cite[Section 3]{Bogomolov}}]\label{BG-th}
        If a natural number $n>2$ divides $c_1(\Line)$, there exists a simply connected non-K\"{a}hler 
        compact complex smooth manifold $Q$ and a finite morphism 
	\begin{equation*}
	p\colon Q \to W_F.
	\end{equation*}
	of degree $n^2$. The manifold $Q$ is isomorphic to the one described in \textup{\cite{Guan1}}.
\end{theorem}
To conclude, the BG-manifold $Q$ is included in the following diagram:
\begin{equation*}
 \xymatrix{
 Q \ar[rrd]_{\Phi}\ar[r]^{p\hspace{10pt}} & W/F \ar[rd]^{\Pi} & W\ar[l]_{\hspace{10pt}q} \ar@{^{(}->}[r]
 \ar[d]^{\pi_0}& S^{[n]}\ar[d]^{\pi^{[n]}} \ar[r]^{\delta\hspace{10pt}} & \Sym^n(S) \ar[d]^{\Sym^n(\pi)} & S\times\dots\times S \ar[d]^{\pi\times\dots\times\pi} \ar[l]_{\alpha} \\
 && \p^{n-1} \ar[d] \ar@{^{(}->}[r] & E^{[n]} \ar[r]^{\cong\hspace{9pt}} \ar[d]^{\Sigma} & \Sym^n(E) & E\times\dots\times E \ar[l]_{\alpha_E}\\
 && 0\ar@{^{(}->}[r] & E &&
 }
\end{equation*}

\begin{remark} \label{Guan-constr}
The original construction of BG-manifolds is due to Guan \cite{Guan3, Guan2} and it relies on nilmanifolds. 
In his papers Guan used the two ways of how one may obtain a compact holomorphic symplectic manifold. The first one is well-known, 
it is the holomorphic symplectic reduction. The second is introduced in \cite{Guan2}. 

\begin{itemize}
	\item[(i)] For any given compact holomorphic symplectic manifold $M$ with the Albanese map  $M \to A$ 
	let $V$ be a subspace of the pull-backed holomorphic 1-forms from $A$ such that the holomorphic symplectic structure vanishes on $V$. 
	Let $G$ be a Lie group of holomorphic vector field dual to $V$. Then one might get a compact holomorphic symplectic manifold by the 
	smooth covering of the quotient of a fiber of $M$ over a point of $A$ by the Lie group $G$.
	
	\item[(ii)] For the holomorphic symplectic nilmanifold $\mathcal{N}$ one can try to find a faithful representation of finite subgroup $\mathrm{S}$ of $\Aut(G)$, where $G$ is the Lie group corresponding to the nilmanifold $\mathcal{N}$. Suppose this representation preserves the complex structure as well as the symplectic structure, and suppose the set \[D = \big\{\ n \in \mathcal{N}\ :\ s(n)=n\ \text{for some}\ s \in \mathrm{S}\ \big\} \] has only codimension 2 irreducible components and $\pi_1(\mathcal{N})^{\mathrm{S}}=1$. Then the desingularization of $\mathcal{N}/\mathrm{S}$ might be the simply-connected non-K\"ahler holomorphic symplectic manifold.
\end{itemize}

The first method was used by Guan in \cite{Guan3} for the Hilbert scheme $S^{[n]}$ of length $n$ of any Kodaira-Thurston surface $S$. 
The covering of orbifold was obtained using its isomorphism to the quotient of a nilmanifold by a subgroup of an automorphism group. 
This method is close to the Bogomolov's one described above.

The second method was applied in \cite{Guan2} where the author started with nilmanifold $M_{n,t}$ given by structure equations, where $n \in \N, t \in \ZZ$. 
Then he got a deformation family $Q_{n,t}$ of simply-connected holomorhic symplectic manifolds which are non-K\"ahler of dimension $2n-2$ paramtrized by 
integer parameter $t$. This is analogous to the existence of different Kodaira surfaces parametrized by $c_1(L)$. This construction 
generalizes the one in \cite{Guan3}.

In particular, the manifold $Q_{n-1,1}$ coincides with $Q$ constructed in \ref{BG-th}. 
Moreover, from Guan's construction it follows that in dimension 2 the manifold $Q$ is a K3 surface since $\mathcal{N}$ is abelian in this case.

\end{remark}

\begin{example}\label{ex: n=2}
Consider the case $n=2$. Then the manifold $R/F$, where $R:=(\Sym^2\pi)^{-1}(0)$, is the fiber of a K\"ahler torus of dimension $2$ (see Lemma \ref{lemma_algebraic_fibers_over_curve})
and the induced action of $\SG_2$ is an involution on this torus. Therefore, $W_F$ is bimeromorphic to Kummer K3 and $Q$ 
is its cover of degree $4$. By construction $Q$ is smooth simply-connected K\"ahler holomorphic symplectic manifold; thus, it should be also a K3 surface.
\end{example}

\begin{proposition}\label{lemma_fiber_of_Phi_abelian}
 A generic fiber of $\Phi\colon Q\to \p^{n-1}$ is an abelian variety; in particular, it is connected.
\end{proposition}
\begin{proof}
 To show this, we need to recall the construction of the finite cover $p\colon Q \to \Base$.
 If we consider $S$ as a real manifold of dimension $4$, there is an action of $\Sph^1$ on $S$ 
 induced from the action on~$\L^*$. This free action gives us a structure of a fiber space:
 \begin{equation*}
  \psi\colon S \to T,
 \end{equation*}
 where $T = S/\Sph^1$ is a real $3$-dimensional torus.
 There is a map from the Hilbert scheme of $S$ to the Hilbert scheme of $T$:
 \begin{equation*}
  S^{[n]} \xrightarrow{\Sym^n(\psi)\circ \delta} \Sym^n T \xrightarrow{\Sigma_T} T.
 \end{equation*}
 Here the map $\Sigma_T$ is induced by the map $\Sym^n T \to T$ which maps a set of $n$ points in $S$ to their sum.
 By \cite[Proposition 2.1]{Bogomolov} the map $\Sigma_T$ is well-defined and the fiber $M$ 
 of the composition $\Sigma_T\circ \psi^{[n]}$ is smooth and irreducible.
 
 By \cite[Lemma 3.8]{Bogomolov} the finite cover $p\colon Q\to W_F$ factors in the following way:
 \begin{equation*}
  \xymatrix{Q \ar[r]_{\tau} \ar[rrd]_{\Phi}  \ar@/^1pc/[rr]^{p} & M/\Sph^1 \ar[r] \ar[rd]^{\Psi} & W_F\ar[d]^{\Pi}\\
  &&\p^{n-1}  
  }
 \end{equation*}
 By construction, fibers of $\Psi$ are connected. Now we fix a general point $x\in \p^{n-1}$ and 
 consider fibers $Q_x$ and $(M/\Sph^1)_x$ of maps $\Phi$ and $\Psi$. By \cite[Remark 3.5]{Bogomolov} we see
 that the generator of the fundamental group $\pi_1(M/\Sph^1)$ is induced by an element in $\pi_1((M/\Sph^1)_x)$.
 Moreover, by \cite[Corollary 3.7]{Bogomolov} we see 
 \begin{equation*}
  \pi_1(Q_x) \subset \pi_1((M/\Sph^1)_x)
 \end{equation*}
 and the index of the subgroup equals $n$. Therefore, we deduce that $Q_x$ is a connected finite unramified  cover
 of $(M/\Sph^1)_x$. Moreover, by \cite[Remark 3.9]{Bogomolov} we see that $Q_x$ is a finite unramified cover
 of a generic fiber of $W_F$ which is isomorphic to an abelian variety $\Pi^{-1}(x)\cong F^{n-1}$. 
 By the Serre-Lang theorem \cite[Chapter IV]{Mumford}, we conclude
 that $Q_x$ is an abelian variety as well.
\end{proof}

\section{The algebraic reduction of BG-manifolds}\label{sec: W/F}

The goal of this Section is to prove Theorem A. We start with the algebraic reduction
\[
\pi: S\overset{F}{\longrightarrow} E
\]
of a Kodaira surface $S$ (with general fiber denoted by $F$). Let $\epsilon\colon E^n \to E$ be the map which sends $n$ points on $E$ to their 
sum, and $A=\ker\epsilon=\epsilon^{-1}(0)$ be the fiber of this map over $0\in E$. The elliptic curve  $F$ acts on~\mbox{$(\pi^n)^{-1}(A)\subset S^n$} fiberwise diagonally.

\begin{equation*}
\xymatrix{
	(\pi^n)^{-1}(A)\ar@{^{(}->}[r]\ar[d]_{\pi^n} & S^n\ar[d]_{\pi^n} & \\
	A\ar@{^{(}->}[r] & E^n\ar[r]^{\epsilon} & E\ni 0
}
\end{equation*}
Denote by $X$ the quotient $(\pi^n)^{-1}(A)/F$ of this action.
The abelian variety $A$ is isomorphic to~$E^{n-1}$ and $\epsilon$ induces a map
\begin{equation}\label{eq_X_to_A}
\eta \colon X\to A.
\end{equation}
These two manifolds are connected as follows.
\begin{proposition}\label{prop: K(X)=K(A)}
	The map $\eta$ is the algebraic reduction of $X$ (in particular, $\Mer(X)\cong\Mer(A)$).
\end{proposition}
\begin{proof}
	In view of \eqref{eq_X_to_A} we have an injection $\Mer(A)\hookrightarrow \Mer(X)$. 
	Denote the general fiber of $\eta$ by $B$. By construction, $B$ is isomorphic to an abelian variety $F^n/F \cong F^{n-1}$. 
	There is an action of~$B$ on $\Mer(X)$. Set 
	\begin{equation*}\label{eq: B'}
	B_0=\max\big \{G\subset B\ :\ \Mer(X)^{G} = \Mer(X)\big \}.
	\end{equation*}
	There is a standard action of the symmetric group $\SG_n$ on $S^n$ and, therefore, on $X$. By the definition of~$B_0$ its action
	and the action of $\SG_n$
	commute, so $B_0$ is a~\mbox{$\SG_n$-in}variant abelian subvariety in $F^n/F$.

	Any abelian subvariety 
	$M\subseteq F^{n}$ can be uniquely defined by its tangent subspace $T_{M,0} \hookrightarrow T_{F^n, 0}$. 
	If $M$ is \mbox{$\SG_n$-in}variant, then the same is true for $T_{M,0}$. There are only two irreducible subrepresentations of the representation~$T_{F^n, 0}$
	of the group $\SG_n$. Namely, these are a trivial diagonal subrepresentation and a standard $(n-1)$-dimensional subrepresentation. 
	Consider the abelian subvariety corresponding to the trivial subrepresentation. Its image in $B$ equals $0$.
	The abelian subvariety corresponding to the standard subrepresentation maps surjectively to $B$. Thus,
	all \mbox{$\SG_n$-in}variant subvarieties of $B$ either coinside with $B$ or vanish.
	
	If $B_0 = B$ we get that $B$ acts transitively on fibers of $\eta$ and all meromorphic functions on~$X$ are constant on fibers 
	of $\eta$. Thus, $\Mer(X)$ is isomorphic to $\Mer(A)$. 
	
	If $B_0 = 0$, then we prove below that $a(X)=\dim(X)$. In order to show this consider the module $\D$ of $\Mer(A)$-linear derivations on~$\Mer(X)$:
	\begin{equation}
	 \D = \mathrm{Der}_{\Mer(A)}(\Mer(X)).
	\end{equation}
        Then by Noether normalization lemma, $\tr\deg_{\Mer(A)}(\Mer(X)) = \dim_{\Mer(A)}(\D)$. Consider a $1$-parameter subgroup $\Gamma = \{\gamma_t\}$ in $B$; this subgroup
        defines a derivation in $\D$. The dimension of the abelian variety $B$ equals $n-1$, so we fix its generators $\gamma_1,\dots, \gamma_{n-1}$ and
        denote by $\Gamma_i = \langle t\cdot\gamma_i\rangle$ a 1-parameter group in $B$ and by $D_i$ the associated derivation in $\D$.
        Assume that these derivations are linearly dependent over the field $\Mer(A)$: there exist meromorphic functions $f_1,\dots,f_{n-1}$ in~$\Mer(A)$ such that
        \begin{equation}\label{eq_linear_dependency}
         \sum_{i=1}^{n-1} f_i D_i = 0.
        \end{equation}
        By construction of $D_i$ for any function $g\in \Mer(X)$ such that the fiber $\eta^{-1}(a)$ is not a pole of $g$ we have:
        \begin{equation*}
         \left.\left(D_i(g)\right)\right|_{\eta^{-1}(a)} = D_i(g|_{\eta^{-1}(a)}).
        \end{equation*}
        There is a non-empty open subset $\U\subset A$ such that for any point $a\in \U$ it is not a pole of $f_i$ 
        for all~\mbox{$i = 1,\dots, n-1$}  and $f_j(a)\ne 0$ for some $j$. Then for each $a\in \U$ and any $g\in \Mer(X)$ 
        such that $\eta^{-1}(a)$ is not a pole of $g$  we have
        \begin{equation*}
         \sum_{i=1}^{n-1} f_i(a) D_i(g|_{\eta^{-1}(a)}) = 0.
        \end{equation*}
        Construct a group $\Gamma = \langle t\cdot\gamma\rangle$  in such a way:
        \begin{equation*}
         \gamma = \sum_{i=1}^{n-1} f_i(a) \gamma_{i}.
        \end{equation*}
        The derivation $D$ which corresponds to $\Gamma$ vanishes on all functions $g|_{\eta^{-1}(a)}$. This implies that for all
        such functions $\gamma\cdot g|_{\eta^{-1}(a)} = g|_{\eta^{-1}(a)}$. Denote by $B'_a$ a minimal abelian subvariety of $B$ 
        which contains~$\Gamma$. Then $\Gamma$ is Zariski dense subset of $B'_a$; and therefore, for each $b\in B'_a$ we have 
        \begin{equation*}
         b\cdot g|_{\eta^{-1}(a)} = g|_{\eta^{-1}(a)}.
        \end{equation*}
        Thus,  there is a family of subgroups $B'_a\subset B$ for each $a\in\U$. Since a choice of a subgroup is a choice of a sublattice,
        i.e. of discrete data, all these subgroups coincide, denote by $B'$ this subgroup of $B$. By construction~\mbox{$\Mer(X)^{B'} = \Mer(X)$}. This contradicts
        to the assumption; thus, there is no linear dependency~\eqref{eq_linear_dependency} and $a(X) = \dim(X)$.

	Thus, $X$ is a Moishezon manifold. 
	The abelian varieties $A$ and $B$ do not contain projective lines, since the embedding of $\p^1$ should factor through the Albanese variety 
	of $\p^1$ which is a point. Therefore, $X$ does not contain a projective line either. Since $X$ carries an algebraic space structure 
	(see Remark \ref{rem: Artin}), \cite[Corollary, 1.4.6]{BCHM} implies that $X$ is projective and, consequently, K\"ahler. 
	
	Consider the embedding of an elliptic curve $E$ to $A \subset E^n$ where the point $x$ maps to $(x,-x,0,\dots,0)$. 
	Denote $(\pi^n)^{-1}(E)/F$ by $X_E$, it is isomorphic to $S\times_E S\times F^{n-3}$. Then $X_E$ maps surjectively 
	to a Kodaira surface; thus, $X_E$ is not in the Fujiki class~$\mathcal{C}$ (see \cite{class_C}).
	In particular, this implies that $X_E$ is not K\"ahler. Then $X$ contains a non-K\"ahler submanifold and we get a contradiction. 
\end{proof}
\begin{proposition}\label{proposition_k_W/F}
 The map $\Pi\colon W_F\to \p^{n-1}$ is the algebraic reduction.
\end{proposition}
\begin{proof}
	The manifold $W/F$ is a resolution of singularities of $X/\SG_n$. Then Proposition~\ref{prop: K(X)=K(A)} implies
	\begin{equation*}
	\Mer(W/F) \cong\Mer(X)^{\SG_n}\cong \Mer(A)^{\SG_n}
	\end{equation*}
	Since by Abel theorem $A^{\SG_n} = \p^{n-1}$, we get the result.
\end{proof}
\begin{corollary}\label{corollary_alg_red_of_Q}
 The map $\Phi\colon Q\to \p^{n-1}$ is the algebraic reduction, and $\Mer(Q) = \Mer(W_F)$.
\end{corollary}
\begin{proof}
 Since $Q$ is a finite cover of $W_F$ their algebraic dimensions are equal
 \begin{equation*}
  \a(W_F) = \a(Q).
 \end{equation*}
 Denote by $\xi\colon Q \to Q_0$ the algebraic reduction of $Q$; by definition of the algebraic reduction there exists a finite rational map:
 \begin{equation*}
  p_0\colon Q_0 \to \p^{n-1}
 \end{equation*}
 However, by Lemma \ref{lemma_fiber_of_Phi_abelian} for a dense set of points $x$ we know that 
 $\Phi^{-1}(x) = (p_0\circ\xi)^{-1}(x)$ are connected manifolds. Therefore, $p_0$ is a map of degree 1 
 and $\p^{n-1}$ is the algebraic reduction of $Q$.
\end{proof}

\section{Classes \texorpdfstring{$\mathcal{N}_k$}{N\_k} and submanifolds in products of Kodaira surfaces}\label{sec: class N}

This section is mostly auxiliary for our future discussion. We are going to introduce some classes of complex spaces that will be useful for distinguishing non-algebraic manifolds. We then use this notion to classify all algebraic subvarieties in products of Kodaira surface.
\begin{definition}
 Let $k\in\ZZ_{\geqslant 0}$. We call $\N_k$ the class of {complex spaces} $X$ such that there exists a subspace $M\subseteq X$ with $\dim(M) - \a(M)\geqslant k$.
\end{definition}
\begin{lemma}\label{lemma_first_property_Nk}
 The power of Kodaira surface $S^k$ is in the class $\N_k$ and 
 there is no submanifold $X\subseteq S^k$ such that $X\in \N_{k+1}$.
\end{lemma}
\begin{proof}
 The first assertion is due to Lemma \ref{lemma_alg_red_of_product}.
 Now consider a submanifold $X\subseteq S^k$ and its image under the map $\pi^k\colon S^k \to E^k$; denote $\pi^k(X)=Z$. Since $Z$ is algebraic, one has $\dim(Z)=\a(Z)$. But $\a(Z)\leqslant\a(X)$ by Lemma \ref{lem: alg dimension properties}, hence
\begin{equation*}
  \dim(X) - \a(X) \leqslant \dim(X) - \dim(Z) \leqslant k.
 \end{equation*}
Therefore, any $X\subseteq S^k$ is not in the class $\N_{k+1}$. 
\end{proof}
\begin{remark}
 The proof of Lemma \ref{lemma_alg_red_of_product} shows also that if the algebraic reduction 
 $\xi\colon X\to X_0$ of a manifold~$X$ is a regular map and all its fibers $\xi^{-1}(x)$ for 
 $x\in X_0$ are of same dimension, then $X\in\N_k$ where~\mbox{$k = \dim(X) - a(X)$}. 
\end{remark}

{
\begin{remark}
	Clearly, the classes $\N_k$ form the descending chain
	\[
	\N_0\supset\N_1\supset\N_2\supset\ldots\supset\N_k\supset\N_{k+1}\supset\ldots.
	\]
	By (\ref{eq: alg dimension bound}), the class $\N_0$ consists of all complex spaces, and if $X\in \N_k$ with $k>0$, then $X$ is non-algebraic. 
	Moreover, by Lemma \ref{lemma_first_property_Nk} we see that $S^k\in \N_k$ and $S^k\not\in \N_{k+1}$.
	So all classes $\N_k$ are non-empty and all embeddings $\N_{k+1}\subset\N_{k}$ are strict. 
\end{remark}
}

The classes $\N_k$ have the following properties.
\begin{lemma}\label{lemma_properties_of_Nk}
 {Let $X$ and $Y$ be two complex spaces.} Then the following assertions hold:
 \begin{enumerate}
  \item[$(1)$] {if $Y\in \N_k$ and $f: X\to Y$ is a surjective morphism, then $X\in \N_k$};
  \item[$(2)$] {if $X\in \N_k$, $f:X\to Y$ is a proper morphism, and $d<k$ is the dimension of a general fiber of $f: X\to f(X)$, then $Y\in \N_{k-d}$}.
  \item[$(3)$] if $X\in \N_1$, then $\dim(X)> a(X)$.
 \end{enumerate}
\end{lemma}
\begin{proof}
 We start with (1). Let $Y\in \N_k$. Then there is a subspace $M\subset Y$ such that $\dim(M) - a(M)\geqslant k$. {The set $Z=f^{-1}(M)$ is an analytic subset of $X$, and carries a natural complex subspace structure.} Then by Lemma \ref{lem: alg dimension properties} we have
 \begin{equation*}
  a(Z) \leqslant a(M) + \dim(f|_{Z})=a(M)+\dim Z-\dim M\leqslant\dim Z-k,
 \end{equation*}
 which implies $X\in\N_k$. To prove (2), take $M\subset X$ with $a(M)\leqslant\dim(M)-k$ and set $W=f(M)\subset Y$. 
 {By Remmert's proper mapping theorem, $W$ is an analytic subset of $Y$ (see e.g. \cite[III, Corollary 4.3]{GPR})}. By Lemma \ref{lem: alg dimension properties} one has
 \[
 a(W)\leqslant a(M)\leqslant\dim(M)-k\leqslant\dim W+d-k,
 \]
 so indeed $Y\in\N_{k-d}$.
 
 To prove (3), assume that $\dim(X) = \a(X)$. Then $X$ is a Moishezon space. However, any subspace $M$ of a Moishezon space is Moishezon \cite[Corollary 3.9]{Ueno}, so $\dim(M) = \a(M)$ and $X\not\in \N_1$.
\end{proof}

\subsection{Digression: nilmanifolds and products of Kodaira surfaces}
Consider $n$ possibly different Kodaira surfaces $S_1,\dots,S_n$ fibered over an elliptic curve $\pi_i\colon S_i\to E$. 
Denote by $\S = S_1\times\dots\times S_n$ their product. 
The manifold $\S$ admits a structure of abelian fibration 
\begin{equation*}
\Pcal \colon \S \to E^n.
\end{equation*}
We study the geometry of a fiber of $\Pcal$ over a curve $C$ in $E^n$.
The preimage of $C$ here is isomorphic to a fiber product:
\begin{equation*}
 \Pcal^{-1}(C) = \left(S_1 \times_E C\right)\times_C \dots \times_C \left(S_n \times_E C\right) \xrightarrow{\Pcal} C.
\end{equation*}
In case when $C$ is an elliptic curve, each surface $S_i \times_E C$ is a Kodaira surface; in particular, it is in the class $\N_1$. 
Moreover, for any curve $C$ we have natural surjective morphisms  $q_i\colon S_i \times_E C\to S_i$. By Lemma \ref{lemma_properties_of_Nk}
this implies that $S_i \times_E C$ are also in the class $\N_1$. Thus, all these surfaces are non-algebraic.
Denote by $q^C_i$ the composition of $q_i$ and the projection from $\Pcal^{-1}(C)$
to its multiple $S_i \times_E C$:
\begin{equation}\label{eq_projections}
 q^C_{i}\colon \Pcal^{-1}(C)\to S_i\times_E C\overset{q_i}{\to} S_i.
\end{equation}
Now we are ready to prove that most submanifolds of $\S$ are non-algebraic.
\begin{lemma}\label{lemma_subvar_product_S}
 If $X$ is a submanifold of $\S$, then either it is contained in the fiber of $\Pcal$ over a point of $E^n$ and this fiber is an abelian variety, 
 or $X\in\N_1$.
\end{lemma}
\begin{proof}
 Assume that $\Pcal(X)$ is not a point. 
 Consider a curve $C$ inside $\Pcal(X)$ and the preimage of this curve $M  = \Pcal^{-1}(C)\cap X$.
 As we saw in \eqref{eq_projections}, there are $n$ projections $q^C_i$ from $M$ to Kodaira surfaces $S_i$. 
 Since $\Pcal(M)$ is a curve, there exists a number $1\leqslant i\leqslant n$ such that the 
 map $q^C_{i}\colon M\to S_i$ is surjective. By Lemma \ref{lemma_properties_of_Nk}
 this implies that $M\in \N_1$. Therefore, $X\in \N_1$.
\end{proof}
\begin{remark}
  According to Guan's approach, the BG-manifolds are constructed as resolutions of nilmanifolds $M/S$, where $M$ is a nilpotent Lie group and $S$ is a finite group in $\Aut(M)$ 
  (see Remark \ref{Guan-constr}). 
  Thus, Lemma \ref{lemma_subvar_product_S} reminds the following folklore statement studued by the first author in 1970-s: 
Let~$N/\Gamma$ be the quotient of a nilpotent Lie group $N$ 
  	by its discrete subgroup $\Gamma$. Then any compact complex curve $C\subset N/\Gamma$ is in fact contained in the abelian 
  	subvariety $\CC^n/ \ZZ^{2n}$ where $\CC^n$ is a shift of a subgroup of $N$ and $\ZZ^{2n}\subset \Gamma$ is also a subgroup.
  	
The counterexample to the statement above given by the curve $C$ in abelian variety $A$ and the Kodaira-type manifold $M_{\L}$ with the 
base $A$ constructed by some linear bundle $\L$ (see Section \ref{subsec_Kod-type}), such that this linear bundle is trivial over $C$.
  	
  	Related questions were studied in a work \cite[Section 8]{HW}, where
  	they proved that if $Z$ is compact complex space in the nilmanifold $G/\Gamma$, then there is closed complex subgroup 
  	$H < G$ such that the quotient~$H/\Gamma$ is compact, and after modification of $\Gamma$ the fibration $Z \rightarrow Z/H$ can be extended to the fibration af the ambient homogeneous space.
\end{remark}

\section{Submanifolds of BG-manifolds}\label{sec: submanifolds of BG}

The goal of this section is to give a description on submanifolds of BG-manifolds. Consider the diagonal action of the elliptic curve $F$ on the power of Kodaira surface $S^n$ and the quotient manifold $S^n/F$. Then $\pi^n$ induces the morphism
\begin{equation*}
 \pi_F^n\colon S^n/F \to E^n.
\end{equation*}
Denote by $p_i$  and $P_i$ the projections from products $E^n$ and $S^n$ to their $i$-th components:
\begin{equation*}
 p_i\colon E^n \to E,\ \ \ \
 P_i\colon S^n \to S.
\end{equation*}
To describe submanifolds of $Q$, we first need to describe those of $S^n/F$.

\begin{theorem}\label{thm_classification_for_Sn/F}
 If $X$ is a submanifold of $S^n/F$ and $Z = \pi_F^n(X)$, then we are in one of the following situations:
 \begin{enumerate}
  \item $Z$ is a point and $X$ is a subvariety of the fiber $F^{n-1}\cong (\pi^n_F)^{-1}(Z)$.
  \item $Z$ is a curve and $X$ can be algebraic, K\"ahler or non-K\"ahler depending on $Z$ and the Kodaira surface $S$.
  If $Z$ is smooth we have following possibilities:
  \begin{enumerate}
   \item If there exists $i$ and $j$ such that $\deg(p_i|_Z)\neq \deg(p_j|_Z)$, then $X$ is non-K\"ahler in class $\N_1$.
   \item If for all $i$ and $j$ we have $\deg(p_i|_Z)= \deg(p_j|_Z)$, then $X$ is K\"ahler.
   \item Let $\L$ be the line bundle on $E$ in the construction of the Kodaira surface. 
   If for all $i$ and~$j$  we have $(p_i|_Z)^*\L \cong (p_j|_Z)^*\L$, then $X$ is algebraic. 
   In particular, any variety in the preimage of $Z= \{(x,\dots,x)|\ x\in E\}\subset E^n$ is algebraic.
  \end{enumerate}
	\item $\dim(Z) \geqslant 2$ and $X$ is a non-K\"ahler manifold in the class $\N_1$.
 \end{enumerate}
\end{theorem}

The prove this statement, we need some preparation.  

\subsection{Kodaira-type manifolds} \label{subsec_Kod-type}

As we saw above, the Kodaira surface $S$ belongs to the class $\N_1$; in particular, it is non-algebraic.
Now we are going to show that a construction, similar to construction of $S$, describes a big class of submanifolds in $S^n/F$. On the other hand, frequently this construction leads to a manifold of class $\N_1$.

Fix a complex number $q\in\CC^{*}$ with $|q|>1$, and denote by $q^{\ZZ}$ the following subgroup of $\CC^*$:
\begin{equation*}
 q^{\ZZ} = \left\{q^m:\ m\in\ZZ\right\}.
\end{equation*}
Denote by $F$ the elliptic curve $\CC^*/q^{\ZZ}$.

\begin{definition}
Consider a manifold $X$ and a line bundle $\L$ on $X$.
The group $q^{\ZZ}$ acts on the total space of line bundle without zero-section $\L^*$.
Then we call  
 \begin{equation*}
  M_{\L} = \L^{*}/q^{\ZZ}.
 \end{equation*}
a \emph{Kodaira-type manifold with the base $X$}.
\end{definition}

By construction there is a canonical projection from $M_{\L}$ to $X$:
\begin{equation*}
 \pi_{\L}\colon M_{\L}\to X.
\end{equation*}
Note that if $X$ is an elliptic curve and $c_1(\L)\neq 0$, then the Kodaira-type manifold $M_{\L}$ is a (primary) Kodaira surface 
(see Subsection \ref{subsec: Kodaira}). More generally, if $X$ is a smooth curve then the Kodaira-type manifold over $X$ is non-K\"ahler:
\begin{lemma}\label{lemma_surfaces_of_Kodaira_type}
 Let $C$ be a smooth curve and $\L$ be a line bundle on $C$ with $c_1(\L)\neq 0$. Then $M_{\L}$ is non-K\"ahler and belongs to the class $\N_1$.
\end{lemma}
\begin{proof}
 By \cite[Proposition V.5.3(ii)]{BPV} we have $b_1(M_{\L}) = 2g(C)+1$. This implies that $M_{\L}$ is not K\"ahler.
 Similar to the proof of Lemma \ref{lemma_map_to_Kodaira}, we get that $M_{\L}$ is not bimeromorphic to a compact K\"ahler surface. The Enriques-Kodaira classification implies then that $M_{\L}$ is in the class $\N_1$.
\end{proof}

Given a complex manifold $X$, consider two line bundles $\L_1$, $\L_2$ on $X$, a fixed complex number $q$ in $\CC^{*}$ and the fibered product 
\begin{equation*}
 Y = M_{\L_1}\times_{X} M_{\L_2}.
\end{equation*}
If we endow $Y$ with a the diagonal action of $F$ and consider the quotient manifold, we will get a Kodaira-type manifold.

\begin{lemma}\label{lemma_kodaira_type}
 Assume that $X$ is a manifold, $\L_1$ and $\L_2$ are line bundles on $X$. If $Y$ 
 is a fibered product of Kodaira-type manifolds $M_{\L_1}\times_{X} M_{\L_2}$ with diagonal action of $F = \CC^*/q^{\ZZ}$, then $Y/F$ is also a Kodaira-type manifold and
 \begin{equation*}
  Y/F \cong M_{\L_1\otimes \L_2^{\vee}}.
 \end{equation*}
\end{lemma}
\begin{proof}
 We construct a map from $Y/F$ to $M_{\L_1\otimes \L_2^{\vee}}$, starting with an auxiliary map
 \begin{align*}
  \widetilde{\Theta}\colon Y = \left(\L_1^{*}/q^{\ZZ}\right)\times_{X} \left(\L_2^{*}/q^{\ZZ}\right) &\to \left(\L_1\otimes \L_2^{\vee}\right)^{*}/q^{\ZZ};\\
  \widetilde{\Theta}(p, l_1, l_2) & =  l_1\otimes l_2^{\vee}.
 \end{align*}

 Here $p\in X$ is a point, $l_1$ and $l_2$ are points on the fibers $(\L_1^*)_p$ and $(\L_2^*)_p$
 modulo action of the group~$q^{\ZZ}$ and by $l_2^{\vee}$ we denote the unique linear functional 
 on $(\L_2)_p$ which maps $l_2$ to $1$. The map $\widetilde{\Theta}$ is well-defined. If we change $l_1$ and $l_2$ by another representatives $q^{d_1}\cdot l_1$ 
 and $q^{d_2}\cdot l_2$ in fibers~$(\L_1^*)_p$ and $(\L_2^*)_p$ of the quotient-sets $(\L_1^*)_p/q^{\ZZ}$ and $(\L_2^*)_p/q^{\ZZ}$, then
 \begin{equation*}
  \widetilde{\Theta}\left (p, q^{d_1}\cdot l_1, q^{d_2}\cdot l_2\right ) = q^{d_1-d_2}\cdot(l_1\otimes l_2^{\vee}) = l_1\otimes l_2^{\vee}.
 \end{equation*}
 The last equality holds, since we consider $l_1\otimes l_2^{\vee}$ as an element in the 
 quotient-set $\left(\L_1\otimes \L_2^{\vee}\right)_p^{*}/q^{\ZZ}$. 
 
 Moreover, the map $\widetilde{\Theta}$ invariant by the action of $F$ on $Y$. Since the action of $q^{\ZZ}$ commutes with 
 the map~$\widetilde{\Theta}$, it suffices to show that for any representative $\lambda \in \CC$ of an element of $F$ 
 the action of $\lambda$ commutes with $\widetilde{\Theta}$:
 \begin{equation*}
  \widetilde{\Theta}\left(\lambda\cdot (p,l_1,l_2)\right ) = \widetilde{\Theta}\left (p,\lambda\cdot l_1,\lambda\cdot l_2\right ) = 
  \lambda\cdot l_1\otimes (\lambda\cdot l_2)^{\vee} = \lambda\cdot l_1\otimes\lambda^{-1}\cdot l_2^{\vee} = l_1\otimes l_2^{\vee}.
 \end{equation*}
 Therefore, the map $ \widetilde{\Theta}$ factors through the quotient-manifold $Y/F$ and we get the following map:
 \begin{equation*}
  \Theta\colon Y/F \to M_{\L_1\otimes \L_2^{\vee}}.
 \end{equation*}
 Note that any element of $\left(\L_1\otimes \L_2^{\vee}\right)_p$ can be written as a decomposable tensor $l_1\otimes l_2^{\vee}$.
 Then we can define the map
 \begin{equation*}
  \Theta'\colon  M_{\L_1\otimes \L_2^{\vee}}\to Y/F,\ \ \ \ \
  \Theta'(l_1\otimes l_2^{\vee}) = (p, l_1, l_2).
 \end{equation*}
 As above, we can easily see that this is a well-defined map and that $\Theta'\circ\Theta$ and $\Theta\circ\Theta'$ 
 are identity maps of $M_{\L_1\otimes \L_2^{\vee}}$ and $Y/F$. Thus, $\Theta$ defines an isomorphisms of two manifolds. 
\end{proof}
\begin{remark}
 Lemma \ref{lemma_kodaira_type} implies in particular that $M_{\L}\cong M_{\L^{-1}}$.
\end{remark}
Now we can generalize Lemma \ref{lemma_kodaira_type} to the fibered product of $n$ Kodaira type manifold.
\begin{lemma}\label{lemma_fiber_product_of_Kodaira_type}
 If $\L_1$, $\L_2, \dots, \L_n$ are line bundles on $X$, then we have the following isomorphism:
 \begin{equation*}
 \left. \left(M_{\L_1}\times_X M_{\L_2}\times_X\dots \times_X M_{\L_n}\right)\right/F \cong 
  M_{\L_2\otimes\L_1^{\vee}} \times_X \dots M_{\L_{n-1}\otimes\L_1^{\vee}} \times_X M_{\L_n\otimes\L_1^{\vee}}.
 \end{equation*}
 Here we consider the diagonal action of $F= \CC/q^{\ZZ}$ on $M_{\L_1}\times_X M_{\L_2}\times_X\dots \times_X M_{\L_n}$.
\end{lemma}
\begin{proof}
 Denote $M_{\L_i}$ by $M_i$ and  consider the following map:
 \begin{align*}
  \Upsilon\colon \left(M_1\times_{X}\dots\times_X M_n\right)/F &\to \left(M_2\times_X M_1\right)/F \times_X \dots  \times_X \left(M_n\times_X M_1\right)/F;\\
  \Upsilon (x,m_1,\dots,m_n) &= ((x,m_2,m_1),\dots,(x,m_n,m_1)).
 \end{align*}
 We can easily see that this map is well-defined. Consider 
 an element $((x,m_2,m_{12}),\dots,(x,m_n,m_{1n}))$ in the right hand side manifold. Here $m_{1i}$ is an element in the fiber of $M_1$ over the point $x$.
 Then there exists a unique  set of elements $f_3,\dots,f_n$ such that
 $f_i\cdot m_{i1} = m_{21}$.
 Then we can define an inverse map:
 \begin{equation*}
  \Upsilon^{-1}((x,m_2,m_{12}),\dots,(x,m_n,m_{1n})) = (x, m_{12}, m_2, f_3\cdot m_3, \dots, f_n\cdot m_n).
 \end{equation*}
 Another easy computation shows that this map is also well-defined. This proves that $\Upsilon$ is an isomorphism. 
 The last necessary observation is that $\left(M_i\times_X M_1\right)/F = M_{\L_i\otimes\L_1^{\vee}}$ by Lemma \ref{lemma_kodaira_type}.
\end{proof}

Lemmas \ref{lemma_kodaira_type} and \ref{lemma_fiber_product_of_Kodaira_type} give us a description of fibers 
of $S^n/F$ over any subvariety. To see it consider a subvariety $Z\subset E^n$ and denote by $L_Z$ the fiber of $S^n$ 
over $Z$:
\begin{equation*}
 \xymatrix{L_Z = (\pi^{n})^{-1}(Z) \ar@{^{(}->}[r] \ar[d]^{\pi^n|_{L_Z}} & S^n \ar[d]^{\pi} \\ Z \ar@{^{(}->}[r] & E^n}
\end{equation*}

Denote by $\L_i$ the line bundle $p_i^*\L$ on $Z$ and by $S_i$ the Kodaira-type manifold $S_i = M_{\left(p_i|_Z\right)^*\L}$ with the base $Z$. 
Then $L_Z$ is isomorphic to a fiber product of  Kodaira-type manifolds:
\begin{equation*}
 L_Z = S_1\times_Z S_2\times_Z \dots\times_Z S_n.
\end{equation*}

We shall need the following result in the future.
{
\begin{lemma}[{\cite[Corollaire 2.9, 2.10]{Varouchas}}]\label{lemma: maps of Kahler manifolds}
	Let $X$ be a K\"ahler manifold, $Y$ be a complex analytic space and $\pi: X\to Y$ be a proper surjective morphism. Assume that one of the following conditions hold:
	\begin{itemize}
		\item The fibers of $\pi$ have the same dimension and either $Y$ is normal, or $\pi$ is flat;
		\item $Y$ has no non-trivial analytic subsets different from divisors (e.g. $Y$ is a surface).
	\end{itemize}
Then $Y$ is K\"ahler.
\end{lemma}
}

\begin{lemma}\label{lemma_subvar_of_S^n/F_over_curves}
Assume that $X \subset S^n/F$ and $Z=\pi^n(X)$ is a smooth curve in $E^n$. 
 If there exists two indices $i$ and~$j$ such that $\deg(p_i|Z)\neq \deg(p_j|Z)$,
 then $X$ is non-K\"ahler and belongs to the class $\N_1$.
\end{lemma}
\begin{proof}  
  We can assume that $i=1$ and $j=2$.
  The fiber $L_{Z}$ of $\pi^n$ over a curve $Z$ is a fibered product of Kodaira type surfaces $S_1,\dots, S_n$:
 \begin{equation*}
  L_{Z}  = S_1\times_{Z}\dots\times_Z S_n = M_{\L_1}\times_{Z}\dots\times_Z M_{\L_n}.
 \end{equation*}
 By Lemma \ref{lemma_fiber_product_of_Kodaira_type} we conclude that $L_{Z}/F\cong M_{(\L_2\otimes \L_1^{\vee})} \times_Z\dots M_{(\L_n\otimes \L_1^{\vee})}$.
 Consider the following composition:
 \begin{equation*}
  X\subset L_{Z}/F\cong M_{(\L_2\otimes \L_1^{\vee})} \times_Z\dots M_{(\L_n\otimes \L_1^{\vee})} \xrightarrow{\mathrm{pr}_1} M_{(\L_2\otimes \L_1^{\vee})},
 \end{equation*}
 where $\mathrm{pr}_1$ is a projection to the first component of the product. 
 By Lemma \ref{lemma_surfaces_of_Kodaira_type} we know that $Z$ is an algebraic reduction of $M_{(\L_2\otimes \L_1^{\vee})}$. 
 Since $X$ maps surjectively to $Z$ the image of $X$ in $M_{(\L_2\otimes \L_1^{\vee})}$ coincides with it. 
 {Since $M_{(\L_2\otimes \L_1^{\vee})}$ is not K\"ahler by Lemma \ref{lemma_surfaces_of_Kodaira_type}, we get by 
 Lemma \ref{lemma: maps of Kahler manifolds} that $X$ is not K\"ahler either}. By Lemma~\ref{lemma_properties_of_Nk} 
 we get that $X\in \N_1$ also.
\end{proof}

Lemma \ref{lemma_subvar_of_S^n/F_over_curves} describes a big class of non-algebraic manifolds inside $S^n/F$. 
Now we are going to show that there are some algebraic varieties in $S^n/F$ which map to curves in $E^n$.

\begin{lemma}\label{lemma_algebraic_fibers_over_curve}
 If $Z$ is a smooth curve in $E^n$ such that $\deg(p_i|Z)= \deg(p_j|Z)$ for all $i$ and $j$, then $L_Z/F$ is K\"ahler. 
 Moreover, if all line bundles $\L_i = (p_i|_Z)^*\L$ are isomorphic, then $L_Z/F$ is an algebraic variety.
\end{lemma}
\begin{proof}
 By Lemma \ref{lemma_fiber_product_of_Kodaira_type} the quotient manifold $L_Z/F$ is isomorphic to the following:
 \begin{equation*}
  L_Z/F \cong M_{(\L_2\otimes \L_1^{\vee})} \times_Z\dots \times_Z M_{(\L_n\otimes \L_1^{\vee})}.
 \end{equation*}
 Since $c_1(\L_i\otimes \L_1^{\vee}) = 0$ for all $i$ by \cite[Section V.5]{BPV} we get that $ M_{(\L_i\otimes \L_1^{\vee})}$ are
 complex tori for all $i$. Therefore, $L_Z/F$ is K\"ahler.
 
 If we know that all line bundles $\L_i\otimes\L_1^{\vee}$ are trivial, then by construction  we get
 \begin{equation*}
  M_{\L_i\otimes\L_1^{\vee}} = M_{\OP_Z}\cong Z\times F.
 \end{equation*}
 Then Lemma \ref{lemma_fiber_product_of_Kodaira_type} implies that $L_Z/F \cong Z\times F^{n-1}$ is an algebraic manifold.
\end{proof}
\begin{remark}\label{remark_diagonal}
 Lemma \ref{lemma_algebraic_fibers_over_curve} proves in particular that if we consider a diagonal in $E^n$
 \begin{equation*}
  \Delta = \left\{ (x,x,\dots,x)|\ x\in E\right\}\in E^n,
 \end{equation*}
 then the fiber $L_{\Delta}/F \subset S^n/F$  is an algebraic manifold. 
\end{remark}

Denote by $t_x$ the translation of the abelian variety $E^n$ by the element $x=(x_1,\dots,x_n)\in E^n$:
\begin{equation*}
 t_x \colon E^n\to E^n.
\end{equation*}
Consider an antidiagonal $\nabla = \{(x,-x)|\ x\in E\}$ in the product $E^2$. 
Then depending on the Kodaira surface the fiber over $t_x(\nabla)$ in $S^2/F$ could be algebraic or not algebraic.
\begin{lemma}\label{lemma_antidiagonal}
 If line bundle $\L$ associated to the Kodaira surface $S$ admits a section vanishing in one point then there exists $x = (x_1,x_2)\in E^2$
 such that $L_Z/F\subset S^2/F$ is an algebraic manifold and $Z = t_x(\nabla)$.
\end{lemma}
\begin{proof}
 By construction the fiber $L_Z$ is isomorphic to the following fibered product of the Kodaira-type manifolds:
 \begin{equation*}
  L_Z = M_{t_{x_1}^*\L}\times_Z M_{t_{x_2}^*\left((-1)^*\L\right)}.
 \end{equation*}
 By assumption, there exists a point $p$ on $E$ such that $\L\sim \OP_E(N\cdot[p])$ for some integer $N$. 
 \begin{equation*}
  (t_{x_1}^*\L)\otimes \left(t_{x_2}^*\left((-1)^*\L\right)\right)^{\vee} = \OP_Z(N\cdot[p-x_1] - N\cdot[-x_2-p]) \cong \OP_Z(N\cdot[2p - x_1+x_2] - N\cdot[0]).
 \end{equation*}
 If we fix $x_1$ and $x_2$ such that $x_1 - x_2 = 2p$ , then this bundle will be trivial. 
 Then by Lemma \ref{lemma_kodaira_type} we get that the fiber $L_Z/F$ is isomorphic to an algebraic manifold:
 \begin{equation*}
 L_Z/F = M_{\OP_Z}  = Z\times F.
 \end{equation*}
\end{proof}

\subsection{Proof of Theorem \ref{thm_classification_for_Sn/F}}
We will need the following assertion.
\begin{lemma}\label{lemma_subvar_of_S2_by_F_surj_on_base}
 {Endow $S^2$ with diagonal action of $F$, and consider the induced commutative diagram
\begin{equation*}
\xymatrix{
	S^2\ar@{>>}[r]^{\alpha}\ar@{>>}[d]_{\pi^2} & S^2/F\ar@{>>}[dl]^{\pi_F^2}\\
	E^2 &
}
\end{equation*}
}If $X \subset S^2/F$ and $\pi^2_F(X) = E^2$, then $X$ coincides with $S^2/F$ and $X\in \N_1$.
\end{lemma}
\begin{proof}
 The dimension of $X$ equals either $2$ or $3$. Note that the complex space $S^2/F$ is in the class $\N_1$ by Lemmas \ref{lemma_first_property_Nk} and \ref{lemma_properties_of_Nk}. 
 {Since $\pi^2$ is the algebraic reduction of $S^2$ (see Lemma \ref{lemma_alg_red_of_product}) we get that} $\pi^2_F$ is the algebraic reduction of $S^2/F$, 
 which implies that any divisor $X$ in $S^2/F$ is a preimage of a divisor on $E^2$. But this contradicts to our assumption $\pi_F^2(X)=E^2$, so $\dim(X) = 3$ 
 and $X$ coincides with $S^2/F$.
\end{proof}
\begin{proof}[Proof of Theorem \ref{thm_classification_for_Sn/F}]
 If $Z$ is a smooth curve in $E^n$ and $\deg(p_i|_Z)\neq \deg(p_j|_Z)$ for some $i$ and~$j$, then by 
 Lemma~\ref{lemma_subvar_of_S^n/F_over_curves} the manifold $X$ is non-K\"ahler. Otherwise, by 
 Lemma \ref{lemma_algebraic_fibers_over_curve} the manifold $X$ is K\"ahler. Moreover, if for all $i$ and $j$
 line bundles $(p_i|Z)^*\L\cong (p_j|Z)^*\L$, then by Lemma \ref{lemma_algebraic_fibers_over_curve} the 
 manifold $X$ is algebraic. Remark~\ref{remark_diagonal} and Lemma \ref{lemma_antidiagonal} show situations, 
 when this condition holds.

 If $\dim(Z) \geqslant 2$,  then there exists two indices, say $1$ and $2$, 
 such that $(p_1\times p_2)(Z) = E\times E$.
 Projections to the $i$-th component of the product $P_i\colon S^n \to S$ induce the morphism
 \begin{equation*}
  (P_1\times P_2)_F \colon S^n/F \to S^2/F,
 \end{equation*}
 which maps $X$ to a submanifold $Y$ on $S^2/F$. 
 By the choice of $p_1$ and $p_2$ we get that~$\pi^2_F(Y) = E^2$. 
 By Lemma~\ref{lemma_subvar_of_S2_by_F_surj_on_base} we get {that $Y=S^2/F$ (in particular, it is not K\"ahler) and $Y\in\N_1$.}
 {Then Lemma \ref{lemma_properties_of_Nk} imply that $X$ is in the class $\N_1$}. Take a 
 smooth curve $Z'\subset Z$ such that $\deg(p_1|_{Z'})\neq\deg(p_2|_{Z'})$. As we showed above the intersection $X'= X\cap (\pi^n)^{-1}(Z')$
 is non-K\"ahler. Since it is a submanifold of $X$ this implies that $X$ is non-K\"ahler.
\end{proof}

\subsection{Classification of submanifolds of \texorpdfstring{$Q$}{Q}}
Recall that BG-manifolds $Q$ fit into the following commutative diagram:
\begin{equation}\label{Diagram: BG}
 \xymatrix{
 &W \ar@{^{(}->}[r] \ar[d]^{q} & S^{[n]} \ar[d] \ar[r]^{\delta} & \Sym^n S \ar[d] & S^n \ar[l]_{\alpha} \ar[d]^{r} \\
 Q\ar[rd]_{\Phi}\ar[r]^{p\hspace{10pt}} & W/F \ar[d]^{\Pi} \ar@{^{(}->}[r] & S^{[n]}/F \ar[d] \ar[r]^{\delta_F} 
 & (\Sym^n S)/F\ar[d] & S^n/F\ar[d]^{\pi^n_F} \ar[l]_{\alpha_F} \\
 & \p^{n-1} \ar@{^{(}->}[r] & E^{[n]} \ar[r]^{\cong\hspace{9pt}} & \Sym^n(E) & E^n \ar[l]_{\alpha_E} 
 }
\end{equation}
The vertical arrows here are quotients by the action of $F$ and the horizontal maps in the middle row are induced by the corresponding maps from above.
Denote by $\E\subset Q$ the following preimage of the exceptional locus of $\delta$:
\begin{equation*}
 \E = p^{-1}(q(\Exc(\delta)).
\end{equation*}
The set $\E$ is a divisor of $Q$.
Theorem \ref{thm_classification_for_Sn/F} implies such an assertion on submanifolds of $Q$ outside $\E$.

\begin{proposition}[\bf Theorem B]\label{prop: submanifolds of Q}
\label{thm_classification_for_Q}
 Assume that $X$ is a submanifold of $Q$ and $Z = \Phi(X)$, then:
 \begin{enumerate}
  \item If $Z$ is a point, then $X$ is a subvariety of the fiber $\Phi^{-1}(Z)$, which is an abelian variety for a general point $Z$.
  \item If $Z$ is a smooth curve, then $X$ can be in the class $\N_1$ or Moishezon, depending on $Z$:
  \begin{enumerate}
   \item[$(a)$] If there exists $i$ and $j$ such that $\deg(p_i|_{\widetilde{Z}})\neq \deg(p_j|_{\widetilde{Z}})$, then $X\in \N_1$.
   \item[$(b)$] If for all $i$ and $j$ we have $(p_i|_{\widetilde{Z}})^*\L\cong (p_j|_{\widetilde{Z}})^*\L$ and $X\not\subset \E$, then $X$ is Moishezon.
  \end{enumerate}
  Here by $\widetilde{Z}$ we denote a connected component of the preimage $\alpha_F^{-1}(Z)$.
  \item $\dim(Z) \geqslant 2$ and $X$ belongs to the class $\N_1$.
 \end{enumerate} 
\end{proposition}
\begin{proof}
 We use the results of Theorem \ref{thm_classification_for_Sn/F}. Denote by $\widetilde{X}$ a connected component of $\alpha_F^{-1}(\delta_F(p(X)))$. All components of this preimage 
 are isomorphic; thus, we can choose any of them. Set $\widetilde{Z} = \pi_F^n(\widetilde{X})$. 
 If $Z$ is a point, then $X$ is a submanifold of the fiber $\Phi^{-1}(Z)$; for a general $Z$ this is an algebraic (abelian) variety by 
 Proposition \ref{lemma_fiber_of_Phi_abelian}. If $\dim(Z)\geqslant 2$, then $\dim(\widetilde{Z})\geqslant 2$ and $\widetilde{X}$ is a non-K\"ahler manifold in the class $\N_1$. 
 Therefore, $X$ is also in the class $\N_1$ by Lemma~\ref{lemma_properties_of_Nk}. 
 
 Let $Z$ be a curve. Then $\widetilde{Z}$ is a curve and the manifold $\widetilde{X}$ can be in the class $\N_1$ or Moishezon (in fact, algebraic), 
 according to the cases described in Theorem \ref{thm_classification_for_Sn/F} (2). The same, respectively, will hold for $X$. 
\end{proof}

\section{Degenerate fibers of algebraic reductions}\label{sec: degenerate fibers}

 \subsection{Fibers of \texorpdfstring{$\Phi$}{Phi} and \texorpdfstring{$\Pi$}{Pi}} In this section we study singular fibers of the maps to the projective space from manifolds $\Sym^{n}(S)$, $W$, $\Base$ which appear during the Bogomolov construction.
 Denote by $V$ the fiber $\left(\Sigma\circ\Sym^n (\pi)\right)^{-1}(0)$. This is the image $\delta(W)$, see diagram below.
 \begin{equation*}
  \xymatrix{
  Q\ar[drrr]_{\Phi}\ar[r]^{p} & W/F & W\ar[l]_{\hspace{0.3cm}q}\ar[r]^{\delta|_W} \ar[rd]^{\pi^{[n]}} & V \ar[d]^{\Sym^n (\pi)|_V} \ar@{^{(}->}[r] & \Sym^n(S) \ar[d]^{\Sym^n (\pi)} & \\
  & & & \p^{n-1}\cong \Sigma^{-1}(0)  \ar@{^{(}->}[r] & \Sym^n (E) \ar[r]^{\hspace{0.3cm}\Sigma} & E 
  }
 \end{equation*}
 Fix a point $x\in \Sigma^{-1}(0)$ and denote by $Q_x$, $W_x$ and $V_x$ fibers of $\Phi$, $\pi^{[n]}|_W$ and $\Sym^n (\pi)|_V$ over this point.
 Fibers $V_x$ can be easily described:
 \begin{lemma}\label{lemma_fibers_of_V}
  For any fiber $V_x$ there exists a sequence of integers $k_1\geqslant\dots\geqslant k_l>0$ such that $\sum k_i = n$ and
  \begin{equation*}
   V_x = F^{[k_1]}\times F^{[k_2]} \times\dots\times F^{[k_l]}.
  \end{equation*}
 If the point $x = \{x_1,\dots,x_n\}$ is such that $x_i\neq x_j$ for all $i\neq j$, then $l=n$ and $k_1 = \dots = k_n = 1$.
 \end{lemma}

 Consider the following divisor in $\Sigma^{-1}(0)$:
  \begin{equation}\label{eq_definition_of_D}
   D = \left\{ \{y_1,\dots,y_n\} \in \Sigma^{-1}(0):\ \text{there exist}\ i\neq j \text{ such that } y_i = y_j\right\}.
  \end{equation}

 \begin{lemma}\label{lemma_ind_locus_delta}
  The indeterminancy locus  $(\delta|_{W})^{-1}$ lies in $\Sym^n(\pi)^{-1}(D)$ and for any point $x\in \Sigma^{-1}(0)$
  it intersects~$V_x$ by a proper closed subset.
 \end{lemma}
 \begin{proof}
  The map $\delta$ is a resolution of singularities; in particular, it is isomorphism in a neighborhood of a
  smooth point of $\Sym^{n}(S)$. 
 \end{proof}

\begin{lemma}\label{lemma_fibers_of_WF}
 For any point $y\in \Sigma^{-1}(0)\setminus D$ we have~\mbox{$W_y\cong F^n$} and~\mbox{$W_y/F\cong F^{n-1}$.}
 If $y\in D$, then both varieties~$W_y$ and~$W_y/F$ are not birational to abelian varieties.
\end{lemma}
\begin{proof}
 By \eqref{eq_definition_of_D} a point $y = \{y_1,\dots,y_n\}\in \Sigma^{-1}(0)$ is not inside $D$,
 if for all indices $i\neq j$ we have $y_i\neq y_j$.
 By Lemma \ref{lemma_ind_locus_delta} the fiber $W_y$ is isomorphic to $V_y$. Lemma \ref{lemma_fibers_of_V} implies that $W_y\cong F^n$.
 
 If $y\in D$, then we have an isomorphism  $V_y\cong F^{[k_1]}\times F^{[k_2]} \times\dots\times F^{[k_l]}$, 
 where~$k_1\geqslant 2$. Thus, there is a structure of a fiber space on $V_y$ induced by Abel-Jacobi maps:
 \begin{equation*}
  P\colon V_y\to F^l.
 \end{equation*}
 All fibers of $P$ are isomorphic to $\p^{k_1-1}\times\dots \times\p^{k_l-1}$. This variety is not birational to an abelian variety
 since its Albanese dimension equals $l<n$. 
 
 Since the map $\delta|_W$ is birational and its exceptional locus does not contain fibers of $\pi^{[n]}$,
 we get that  the fiber $W_y$ is a union of several irreducible manifolds $W_{y_1}\cap\dots W_{y_s}$ 
 and $W_{y_1}$ is birational to $V_y$. Thus, $W_{y_1}$ is not birational to an abelian variety.
 
 All maps from the projective elliptic curve $F$ to the affine group $\mathrm{PGL}(k_1,\mathbb{C})\times\dots\times\mathrm{PGL}(k_l,\mathbb{C})$ 
 are constant. Therefore, the structure of $\left(\p^{k_1-1}\times\dots \times\p^{k_l-1}\right)$-fiber space remains on $q(W_{y_1})$. 
 Then the fiber $W_y/F$ contains an irreducible component which is not birational to an abelian variety.
\end{proof}

\begin{lemma}\label{lemma_D_nonrational}
  The subset $D$ is a non-rational divisor in $\Sigma^{-1}(0)$.
 \end{lemma}
 \begin{proof}
  Define the following fiber space:
  \begin{equation*}
  \xymatrix{
  \widetilde{D} \ar[d] \ar[r] & E\ar[d]^{\cdot 2}\\
  E^{[n-2]} \ar[r]^{\Sigma} & E
  }
 \end{equation*}
 Any point of $\widetilde{D}$ can be associated with a set $(\{x_1,\dots, x_{n-2}\}, x)$, where $x$ is a point with property:
 \begin{equation*}
  2\cdot x = x_1+\dots+x_{n-2}.
 \end{equation*}
 There is a regular map from $\widetilde{D}$ to $E^{[n]}$:
 \begin{align*}\label{eq_iota}
  \iota\colon \widetilde{D} \to E^{[n]}&&
  \iota(\{x_1,\dots, x_{n-2}\}, x) = \{-x,-x,x_1,\dots,x_{n-2}\}.
 \end{align*}
 The image of the map $\iota$ coincides with $D$. Moreover, the map $\iota$ is birational to the image.
 Therefore, $D$ is birational to $\widetilde{D}$, it is an irreducible subvariety of codimension $1$ in $\Sigma^{-1}(0)$.
 Finally, the Albanese variety of $\widetilde{D}$ equals $E$ since the fibers of the map to $E$ are projective spaces.
 Then we get the following equality:
 \begin{equation*}
  h^{1,0}(\widetilde{D}) = \dim(\mathrm{Alb}(\widetilde{D}) = \dim(E) = 1.
 \end{equation*}
 Since $h^{1,0}$ is a birational invariant of a K\"ahler manifold, we get that $D$ is not rational.
 \end{proof}

 \subsection{\texorpdfstring{$D$}{D} as a dual variety}
 Let $x_0$ be the identity element for the group structure on $E$. If $n\geqslant 3$, then the line bundle $\OP_E(nx_0)$ is
 very ample; denote $W = H^0(E,\OP_E(nx_0))^{\vee}$, it is an $n$-dimensional vector space. Then we have an embedding:
 \begin{equation*}
  E\hookrightarrow \p(W)\cong \p^{n-1} 
 \end{equation*}
 With each hyperplane $H\subset \p^{n-1}$ we can associate a point on a Hilbert scheme of $E$.
 \begin{lemma}\label{lemma_duality}
  If $n\geqslant 3$, then $E$ is not contained in a hyperplane $H\subset \p^{n-1}$. Each $H$ intersects $E$ by the set of $n$ 
  points counted with multiplicities $E\cap H = \{x_1,\dots, x_n\}$; then we have a map:
  \begin{align*}
   \rho\colon \p\left(W^{\vee}\right) \to E^{[n]} \\
   \rho(H)= \{x_1,\dots,x_n\}.
  \end{align*}
 The image of $ \rho$ is contained in $\Sigma^{-1}(0)\subset E^{[n]}$.
  The map $\rho$ induces an isomorphism:
  \begin{equation*}
   D \cong \left \{\ H\subset\PP(W^{\vee}):\ H\ \text{is tangent to}\ E\ \text{in a point}\ x\ \right \}\subset \PP(W^{\vee}).
  \end{equation*}

 \end{lemma}
 \begin{example}\label{example_sextic_with_cusps}
 If $n = 3$, then by Lemma \ref{lemma_duality} we see that $D\subset \p\left(W^{\vee}\right)$ is a dual curve to a plane cubic curve $E\subset \p^2$. 
 Thus, $D$ is a plane curve of degree $6$ with $9$ cusps. In case when $n\geqslant 4$ the variety $D$ is a natural generalisation of the dual variety to
 a curve $E$ in the projective space.
 \end{example}
  Now we introduce some notation. For each point $\{x_1,\dots,x_n\}$ in $\Sigma^{-1}(0)\subset E^{[n]}$ denote by $H_{x_1,\dots,x_n}$ a hyperplane in $\p(W)$ such that
  \begin{equation*}
   H_{x_1,\dots,x_n} \cap E = \{x_1,\dots,x_n\}.
  \end{equation*}
  The hyperplane $H_{x_1,\dots,x_n}$ defines a point in $\p\left(W^{\vee}\right)$.
  Fix a subspace $U$ of dimension $n-3$ inside the hyperplane $H_{x_1,\dots,x_n}\subset \p(W)$. 
  Denote by $L$ a line in $\p\left(W^{\vee}\right)$ passing through the point $x$ on $D$ corresponding to $U$:
  \begin{equation*}
   H_{x_1,\dots,x_n}\in L = \left \{ H\in \PP(W^{\vee}):\ U\subset H\ \right \}.
  \end{equation*}
  Now fix a subspace $U$ which does not intersect $E$. 
  Consider the projection from such $U$ to a general line $\p^1\subset \p(W)$:
  \begin{equation*}
   \pi_U\colon \p(W)\dashrightarrow \p^1.
  \end{equation*}
  By the choice of $U$ the restriction $\left.\pi_U\right|_E \colon E\to \p^1$ is a regular finite map of degree $n$. 
  For a point $y\in \p^1$ the following hyperplane $H_y$ corresponds to a point on $L$:
  \begin{equation*}
   H_y = \langle U,y\rangle \in L \subset \p\left(W^{\vee}\right).
  \end{equation*}
 \begin{lemma}\label{lemma_lines_in_dual}
  The choice of a line $L\subset \p\left(W^{\vee}\right)$ is equivalent to a choice of a subspace $U\subset \p(W)$ of codimension $2$.
  With each $U$ which does not intersect $E$ we can associate a regular map 
  \begin{equation*}
   \pi_U\colon E\to \p^{1}
  \end{equation*}
  of degree $n$ and ramification divisor $R_U\subset E$ is of degree $2n$.
 \end{lemma}
 \begin{definition}
  Denote by $Z$ the set
  \begin{equation*}
   Z = \left\{H\in \p\left(W^{\vee}\right ) :\ |H\cap E|\leqslant n-2\right \}\subset D\subset\p\left(W^{\vee}\right).
  \end{equation*}
  Here we consider a set-theoretical intersection $H\cap E$. 

 \end{definition}

 \begin{lemma}\label{lemma_general_line_in_dual_space}
 The set $Z$ is a closed and proper subset of \ $D$.
 \end{lemma}
 \begin{proof}
  The subset $Z$ is closed by the definition. To prove its properness take $n-3$ pairwise distinct points~\mbox{$x_1,\dots, x_{n-3}\in E$} and choose $y$ 
  such that $y\not \in \{x_1,\dots,x_{n-3}\}$ and
  \begin{equation*}
   z = \left(-2y-\sum_{i=1}^{n-3} x_i\right) \not\in \{x_1,\dots,x_{n-3},y\}
  \end{equation*}
 Then $\{x_1,\dots,x_{n-3},y,y,z\}$ is a point on $D$, but it does not lie in $Z$. Thus, $Z$ does not coincide with $D$. 
 \end{proof}
 Consider a divisor $D$ in $\p^N$ of degree $d$. We say that a point $x\in D$ is of multiplicity $m$, if for a general line $l\subset \p^N$ 
 passing through $x$ we have:
 \begin{equation*}
  |l\cap D| = d-m+1.
 \end{equation*}

 \begin{proposition}\label{prop_singularities_of_D}
  The degree of the divisor $D$ in $\p\left(W^{\vee}\right)$ equals \ $2n$.
  Fix a point $x=\{x_1,\dots,x_n\}\in D$ such that $|\{x_1,\dots,x_n\}| = r$. Then $x$ is a point on $D$ of multiplicity $n - r$.
 \end{proposition}
 \begin{proof}
 By Lemma \ref{lemma_general_line_in_dual_space} we can choose a line $L$ in $\p\left(W^{\vee}\right)$ which does not intersect $Z$. By Lemma \ref{lemma_lines_in_dual} this  line  $L$ corresponds to a subspace $U\subset \p(W)$. Since $L$ does not intersect $Z$, each hyperplane
 $H$ which contain $U$ is not a bitangent to $E$. Therefore, the ramification divisor $R_U$ of $\pi_U$ consists of $2n$ distinct points.
 By the generality of the choice of $L$ we see that $\deg(D)=\deg(R_U)=2n$ by Lemma \ref{lemma_lines_in_dual}.
 
 Consider a point $x$ as in the assumption of the lemma. 
 If $r=n-1$, then a general line $L$ through~$x$ does not intersect $Z$ by Lemma \ref{lemma_general_line_in_dual_space} and the multiplicity of $x$ equals $1$.
 
 If $r<n-1$, then $x\in Z$. Since $Z$ is of codimension $2$, a general line $L$ passing through $x$ does not intersect $Z$ in any point except $x$.
 Then we conclude that the point $x$ is of multiplicity $n - r$.
 \end{proof}
 
Denote by $\Z_k$ the set of singular points on $D\subset \p\left(W^{\vee}\right)$ of multiplicity $\geqslant k$. These sets give a stratification of $D$:
\begin{equation*}
  \emptyset = \Z_{n} \subsetneq \Z_{n-1} \subsetneq \dots \subsetneq \Z_{2} \subsetneq \Z_{1} = D.
\end{equation*}

\begin{proposition}\label{prop_points_in_D}
 The set $\Z_{n-1}\subset D$ is a finite set of singular points of maximal multiplicity. It does not lie on any hyperplane in~$\p\left(W^{\vee}\right) \cong \Sigma^{-1}(0)$.
\end{proposition}
\begin{proof}
 By Lemma \ref{lemma_duality} and Proposition \ref{prop_singularities_of_D} each point of $\Z_{n-1}$ 
 corresponds to a point $H_{x,\dots,x}$ of $\p\left(W^{\vee}\right)$, where $x$ is a point in the group of $n$-torsion $E[n]\cong \ZZ/n\times\ZZ/n$. Thus, 
 \begin{equation*}
  |\Z_{n-1}| = n^2.
 \end{equation*}
 
 There is a standard action of a group of $n$-torsion $E[n]$ on the projective space $\p\left(W^{\vee}\right)$. 
 For each~\mbox{$\xi\in E[n]$} and a point $H_{x_1,\dots,x_n}$ in $\p\left(W^{\vee}\right)$ we have:
 \begin{equation}\label{eq_action_of_En}
  \xi\cdot H_{x_1,\dots,x_n} =  H_{x_1+\xi,\dots,x_n+\xi}.
 \end{equation}
 This action induces the action of $E[n]$ on $\p(W)$ which leaves the curve $E\subset \p(W)$ invariant and acts on it by 
 translation with $n$-torsion points. By \cite[Theorem I.2.5.]{Hulek} this projective action is induced by the linear action of the standard representation of the finite Heisenberg 
 group $\H_n$ on the~\mbox{$n$-di}\-mensional vector space $W$ (for the definition see \cite[Page 11]{Hulek}). 
 
 Denote by $L_x$ a 1-dimensional subspace of $W^{\vee}$ associated  with a 
 point $H_{x,\dots,x}\in \p\left(W^{\vee}\right)$ for a point~\mbox{$x\in E[n]$}. 
 By \eqref{eq_action_of_En} the action of each element of $\H_n$ maps $L_x$ to $L_y$ for some $y\in E[n]$. Thus, there is a subrepresentation of $W^{\vee}$:
 \begin{equation*}
  L = \left\langle L_x\right\rangle_{x\in E[n]} \subset W^{\vee}.
 \end{equation*}
 However the standard representation $W$ of $\H_n$ and its dual $W^{\vee}$ are irreducible. Then $L = W^{\vee}$
 and there is no hyperplane in $\p\left(W^{\vee}\right)$ which contain all points $H_{x,\dots,x}$ for $x\in E[n]$. 
\end{proof}

\section{Automorphisms of \texorpdfstring{$W_F$}{W\_F} and \texorpdfstring{$Q$}{Q}}\label{sec: auto}

The goal of this paragraph is to prove Theorem C. The definition of the Jordan property is given in Section \ref{sec: introduction}.
Here we will need some other definitions.

\begin{definition}
	Let $\Gamma$ be a group.
	\begin{itemize}
		\item We say that $\Gamma$ is {\it bounded} if there is an integer $\Bound=\Bound(G)\in\ZZ_{\geqslant 0}$ such that each finite subgroup $G\subset\Gamma$ satisfies $|G|\leqslant \Bound$.
		\item We say that $\Gamma$ is {\it quasi-bounded} \cite{BandmanZarhin} or {\it has finite subgroups of bounded rank} \cite{PS-Jordan} if there is an integer $A=A(\Gamma)\in\ZZ_{\geqslant 0}$ such that each finite abelian subgroup of $\Gamma$ is generated by at most $A$ elements.
		\item We say that $\Gamma$ is {\it strongly Jordan} \cite{PS-3folds} if it is Jordan and quasi-bounded.
	\end{itemize}
\end{definition}

\begin{lemma}[{\cite[Lemma 4.1]{PS-Kahler}}]\label{lem: PS}
	Let $X$ and $Y$ be compact complex manifolds, and $\phi: X\to Y$ be a surjective morphism with connected fibers. 
	Let $\{G_i\}$ be a countable family of finite subgroups in $\Bim(X)_\phi$. Then there exists a smooth, irreducible and reduced fiber $F$ of the map $\phi$ of dimension $\dim X-\dim Y$ such that all the groups $G_i$ are embedded into $\Bim(F)$. Moreover, if $\dim(Y)>0$ and we are given a countable family $\Xi$ of proper closed analytic subsets in $Y$, then the fiber $F$ can be chosen so that $\phi(F)$ does not lie in $\Xi$.
\end{lemma}

For a surjective morphism of compact complex manifolds $\phi:X\to Y$ we denote by $\Bim(X)_\phi$ the subgroup 
of $\Bim(X)$ which consists of bimeromorphic selfmaps whose action is fiberwise with respect to $\phi$. Also set
\[
\Aut(X)_\phi=\Aut(X)\cap\Bim(X)_\phi.
\]

\begin{proposition}\label{prop: Aut is Jordan}
	Let $Q$ be a BG-manifold of dimension $2n-2$. Then the group $\Aut(Q)$ is strongly Jordan.
\end{proposition}
\begin{proof}
	By Proposition \ref{prop: bim commutes}, there is a short exact sequence of groups
	\[
	1\to\Aut(Q)_\Phi\to\Aut(Q)\to\Xi_D\to 1,
	\]
	where $\Xi_D\subset\Aut(\PP^{n-1},D)$, the latter group being the group of projective automorphisms preserving $D$. 
	This is a finite group, because by Proposition \ref{prop_points_in_D} it preserves a finite subset of $n^2$ points $\Z_{n-1}\subset D$, which spans $\PP^{n-1}$.
	As in \cite[Corollary 4.2]{PS-Kahler} we note that the group $\Aut(Q)_\Phi$ is strongly Jordan. Namely, assuming the contrary, we use Lemma \ref{lem: PS} to deduce that the group $\Aut(F)$ is not strongly Jordan, where $F$ is a general fiber of $\Phi$. However, we know from Proposition \ref{lemma_fiber_of_Phi_abelian} that $F$ is an abelian variety, so
	\begin{equation}\label{eq: abelian variety auto}
	\Aut(F)\cong F(\CC)\rtimes\Gamma,
	\end{equation}
	where $\Gamma\subset\GL_{2\dim F}(\ZZ)$. By Minkowski's theorem, the latter group is bounded. Therefore, $\Aut(F)$ is strongly Jordan. This is a contradiction. 
	
	Finally, it is easy to see that an extension of a finite group by a strongly Jordan group is again strongly Jordan. 
\end{proof}

\begin{remark}
	Note that every finite subgroup $G\subset\Aut(Q)$ fits into exact sequence
		\[
		1\to G_F\to G\to G_D\to 1,
		\]
		where $G_F$ embeds into $\Aut(F)$ and $G_D$ embeds into $\SG_{n^2}$. By Chermak-Delgado theorem, $G_F$ contains a characteristic abelian subgroup of index at most $|\Gamma|^2$. Therefore, $\Jord(\Aut(Q))\leqslant |\Gamma|^2\cdot|\SG_{n^2}|$. By the results of S. Friedland and W. Feit \cite{Friedland}, one has $|\Gamma|\leqslant |\Ort_{2\dim F}(\ZZ)|=2^{2\dim F}(2\dim F)!$ for $2\dim F>10$ (the remaining cases are somewhat sporadic, see the references therein). Thus, putting $d=\dim Q=2\dim F=2n-2$ we can get a rough estimate
		\[
		\Jord(\Aut(Q))\leqslant 4^d(d!)^2[(d/2+1)^2]!= 2^{4n-4}[(2n-2)!]^2(n^2)!
		\]
		for any BG manifold $Q$ of dimension $d>10$.
\end{remark}

We finish this section by deriving some consequences about the structure of $\Bim(Q)$.

\begin{proposition}\label{prop: bim commutes} 
	Let $Q$ be a BG manifold of dimension $2n-2$ and $W_F$ be its base. Then
	\begin{enumerate}
		\item Any $\varphi\in\Bim(W_F)$ induces $\varphi_0\in \Bim(\p^{n-1})$ 
		such that $\Pi\circ\varphi = \varphi_0\circ\Pi$ and $\varphi_0(D)=D$. Further, there exists the integer $M$ such that
		for any  $\varphi\in\Bim(Q)$ its power~$\varphi^M$ induces~\mbox{$\varphi_0\in \Bim(\p^{n-1})$} and $\varphi_0(D)  = D$.
		\item There exists an exact sequence of groups
		\begin{equation}\label{eq: Bim exact sequence}
			1\to\Bim(Q)_\Phi\to\Bim(Q)\to\Delta\to 1,
		\end{equation}
		where $\Delta$ is a subgroup of the Cremona group $\Bim(\PP^{n-1})\cong\Bir(\PP^{n-1})$. Moreover, the group $\Bim(Q)_\Phi$ is strongly Jordan, and the group $\Bim(Q)$ has finite subgroups of bounded rank.
	\end{enumerate}
	
\end{proposition}
\begin{proof}
	(1) Consider a bi\-me\-ro\-morphic automorphism $\varphi$ of $Q$ and denote by $\varphi_x$ the restriction of $\varphi$
	to a fiber $Q_x$ which is not in the exceptional or indeterminancy locus:
	\begin{equation*}
	\varphi_x\colon Q_x \dashrightarrow Q_y.
	\end{equation*}
	For a general point $x$ both manifolds $Q_x$ and $Q_y$ are abelian; thus, $\varphi_x$ is regular.
	Moreover, $Q_x$ and $Q_y$ are finite covers of abelian varieties $(W_F)_x$ and $(W_F)_y$ and these 
	finite covers induces an isogeny. Therefore, some degree $M$ of $\varphi_x$ induces a morphism between $(W_F)_x$ and $(W_F)_y$
	and $\varphi^M$ induces a bi\-me\-ro\-morphic automorphism of $W_F$. Then it suffices to
	show that any automorphism of $W_F$ fixes a divisor $D$.
	
	Consider a bimeromorphic automorphism $\varphi$ of $W_F$. Since $\Pi$ is the algebraic reduction of $W_F$, we get a bimeromorphic map $\varphi_0$ of the projective space which commutes with $\Pi$. Since $\varphi$ is bimeromorphic, there exists open subsets $\U$ and $\V$ in $Q$, such that $\varphi|_U$ is an isomorphism to~$\V$:
	\begin{equation*}
	\varphi|_{\U}\colon \U \overset{\raisebox{0.25ex}{$\sim\hspace{0.2ex}$}}{\smash{\longrightarrow}} \V.
	\end{equation*}
	Denote by $\U_0$ and $\V_0$ the images of $\U$ and $\V$ in $\p^{n-1}$; these are open subsets in $\p^{n-1}$. 
	For each point~$x\in \p^{n-1}$ denote by $\U_x$ and $\V_x$ dense subvarieties of fibers $\Pi^{-1}(x)\cap \U$ and $\Pi^{-1}(x)\cap \V$. Fix a point $x\in \U_0\cap D$, where $D$ is as in \eqref{eq_definition_of_D}. Then $\varphi(\U_x) \cong \V_{\varphi_0(x)}$; in
	particular, the fiber $\V_{\varphi_0(x)}$ is not birational to an abelian variety. By Lemma \ref{lemma_fibers_of_WF} this implies
	that $\varphi_0(x)\in D$.  
	
	(2) The existence of the short exact sequence (\ref{eq: Bim exact sequence}) is clear, since $\Phi$ is the algebraic reduction of $Q$. To see that $\Bim(Q)_\Phi$ is strongly Jordan we note that $\Bim(F)=\Aut(F)$ and use Lemma \ref{lem: PS}. It remains to notice that both $\Aut(F)$ and $\Bir(\PP^{n-1})$ are quasi-bounded (see \cite[Remark 6.9]{PS-Jordan}), hence the same is true for $\Bim(Q)$.
\end{proof}

\begin{remark}\label{rem: boundednes of Aut}
	It would be interesting to investigate the properties of the groups in (\ref{eq: Bim exact sequence}). For example, we still do not know if $\Bim(Q)$ is a Jordan group, or even whether $\Aut(Q)$ and $\Bim(Q)$ are bounded (note that this is indeed the case for {\it hyperk\"ahler} manifolds, see \cite[Theorem 1.4]{Cattaneo} or \cite[Theorem C]{KurnosovYasinsky}).
\end{remark}

\begin{proof}[Proof of Theorem C]
	Claim (1) was proven in Proposition \ref{prop: Aut is Jordan}. Claim (2) is a part of Proposition \ref{prop: bim commutes}.
\end{proof}

\printbibliography


\end{document}